\documentclass[a4paper, 12pt]{amsart}
\usepackage{fourier}
\usepackage[english]{babel}
\usepackage{amsmath,amsthm,amssymb,amsfonts}
\usepackage{multicol}
\usepackage{todonotes}
 \usepackage{enumitem} 
\usepackage[bookmarks=true,hyperindex,pdftex,colorlinks,citecolor=blue]{hyperref}

\hypersetup{
	colorlinks=true,
	linkcolor=blue,
	filecolor=magenta,   
	urlcolor=cyan,
}

\usepackage{xcolor}
\usepackage[export]{adjustbox}
\usepackage{graphicx}
\usepackage{subcaption}

\usepackage{graphicx}

\usepackage{tikz-cd}

\usepackage{tikz}
\usetikzlibrary{mindmap,backgrounds, calc}
\usetikzlibrary{matrix,chains,scopes,positioning,arrows,fit}
\usetikzlibrary{positioning, shapes, patterns}
\usetikzlibrary{automata} 
\usetikzlibrary{shapes.geometric, arrows, arrows.meta}
\usetikzlibrary{positioning,decorations.markings}

\theoremstyle{definition}
\newcounter{maincoro}

\newtheorem{maincor}[maincoro]{Corollary}

\newtheorem{theorem}{Theorem}[section]
\newtheorem{lemma}[theorem]{Lemma}

\newtheorem{proposition}[theorem]{Proposition}
\newtheorem{corollary}[theorem]{Corollary}

\theoremstyle{definition}
\newcounter{maintheorem}

\newtheorem{mainth}[maintheorem]{Theorem}
\newtheorem{definition}[theorem]{Definition}
\newtheorem{example}[theorem]{Example}

\theoremstyle{remark}
\newtheorem{remark}[theorem]{Remark}
\numberwithin{equation}{section}

\newcommand{\R}{\mathbb{R}}
\newcommand{\C}{\mathbb{C}}
\newcommand{\N}{\mathbb{N}}
\newcommand{\K}{\mathbb{K}}

 \makeatletter
\renewcommand{\tocsection}[3]{%
	\indentlabel{\@ifnotempty{#2}{\bfseries\ignorespaces#1 #2\quad}}\bfseries#3}
\renewcommand{\tocsubsection}[3]{%
	\indentlabel{\@ifnotempty{#2}{\ignorespaces#1 #2\quad}}#3}

\newcommand\@dotsep{4.5}
\def\@tocline#1#2#3#4#5#6#7{\relax
	\ifnum #1>\c@tocdepth 
	\else
	\par \addpenalty\@secpenalty\addvspace{#2}%
	\begingroup \hyphenpenalty\@M
	\@ifempty{#4}{%
		\@tempdima\csname r@tocindent\number#1\endcsname\relax
	}{%
		\@tempdima#4\relax
	}%
	\parindent\z@ \leftskip#3\relax \advance\leftskip\@tempdima\relax
	\rightskip\@pnumwidth plus1em \parfillskip-\@pnumwidth
	#5\leavevmode\hskip-\@tempdima{#6}\nobreak
	\leaders\hbox{$\m@th\mkern \@dotsep mu\hbox{.}\mkern \@dotsep mu$}\hfill
	\nobreak
	\hbox to\@pnumwidth{\@tocpagenum{\ifnum#1=1\bfseries\fi#7}}\par
	\nobreak
	\endgroup
	\fi}
\AtBeginDocument{%
	\expandafter\renewcommand\csname r@tocindent0\endcsname{0pt}
}
\def\l@subsection{\@tocline{2}{0pt}{2.5pc}{5pc}{}}
\makeatother

\DeclareMathOperator{\dist}{dist\,}

\DeclareMathOperator{\co}{co}
\DeclareMathOperator{\aco}{aco}
\DeclareMathOperator{\re}{Re}

\DeclareMathOperator{\id}{Id}

\newcommand{\nn}[1]{{\left\vert\kern-0.25ex\left\vert\kern-0.25ex\left\vert #1 
		\right\vert\kern-0.25ex\right\vert\kern-0.25ex\right\vert}}

\renewcommand{\geq}{\geqslant}
\renewcommand{\leq}{\leqslant}

\newcommand{\NA}{\operatorname{NA}}

\newcommand{\e}{\varepsilon}
\newcommand{\de}{\delta}

\newcommand{\pten}{\ensuremath{\widehat{\otimes}_\pi}}

\newcommand{\eps}{\varepsilon}
\newcommand{\sten}{\ensuremath{\widehat{\otimes}_{\pi_s,N}}}
\newcommand{\psten}{\ensuremath{\widehat{\otimes}_{\pi_s,N}}}

\newcommand{\Loo}{{\bf L}_{o,o}}
\newcommand{\Lpp}{{\bf L}_{p,p}}
\newcommand{\NLoo}{\text{$N$-homogeneous polynomial ${\bf L}_{o,o}$}}

\title[On the SSD of homogeneous polynomials and tensor products]{On the strong subdifferentiability of the homogeneous polynomials and (symmetric) tensor products}

\author[Dantas]{Sheldon Dantas}
\address[Dantas]{Departament de Matem\`{a}tiques and Institut Universitari de Matem\`{a}tiques i Aplicacions de Castell\'o (IMAC), Universitat Jaume I, Campus del Riu Sec. s/n, 12071 Castell\'o, Spain. \newline
	\href{http://orcid.org/0000-0001-8117-3760}{ORCID: \texttt{0000-0001-8117-3760} } }
\email{\texttt{dantas@uji.es}}

\author[Jung]{Mingu Jung}
\address[Jung]{School of Mathematics, Korea Institute for Advanced Study, 02455 Seoul, Republic of Korea \newline
\href{http://orcid.org/0000-0003-2240-2855}{ORCID: \texttt{0000-0003-2240-2855} }}
\email{\texttt{jmingoo@kias.re.kr}}

\author[Mazzitelli]{Martin Mazzitelli}
\address[Mazzitelli]{Instituto Balseiro, CNEA-Universidad Nacional de Cuyo, CONICET, Argentina.}
\email{\texttt{martin.mazzitelli@ib.edu.ar}}

\author[Rodríguez]{Jorge Tomás Rodríguez}
\address[Rodríguez]{Departamento de Matemática and NUCOMPA, Facultad de Cs. Exactas, Universidad Nacional del Centro de la Provincia de Buenos Aires, (7000) Tandil, Argentina and CONICET. \newline
	\href{https://orcid.org/0000-0003-4693-2498}{ORCID: \texttt{0000-0003-4693-2498} } }
	\email{\texttt{jtrodrig@dm.uba.ar}}

\thanks{}


\date{\today}


\begin{document}

\begin{abstract} In this paper, we study the (uniform) strong subdifferentiability of the norms of the Banach spaces $\mathcal{P}(^N X, Y^*)$,\- $X \pten \cdots \pten X$ and $\sten X$. Among other results, we characterize when the norms of the spaces $\mathcal{P}(^N \ell_p, \ell_{q}), \mathcal{P}(^N l_{M_1}, l_{M_2})$, and $\mathcal{P}(^N d(w,p), l_{M_2})$ are strongly subdifferentiable. Analogous results for multilinear mappings are also obtained. Since strong subdifferentiability of a dual space implies reflexivity, we improve some known results in \cite{DimZal, Gon, GJ} on the reflexivity of spaces of $N$-homogeneous polynomials and $N$-linear mappings.
Concerning the projective (symmetric) tensor norms, we provide positive results on the subsets $U$ and $U_s$ of elementary tensors on the unit spheres of $X \pten \cdots \pten X$ and $\sten X$, respectively. Specifically, we prove that $\sten \ell_2$ and $\ell_2 \pten \cdots \pten \ell_2$ are uniformly strongly subdifferentiable on $U_s$ and $U$, respectively, and that 
${\ensuremath{c_0 \widehat{\otimes}_{\pi_s}}} c_0$ and $c_0\pten c_0$ are strongly subdifferentiable on $U_s$ and $U$, respectively, in the complex case.
\end{abstract}

\maketitle
	\tableofcontents

\section{Introduction}

The main aim of this paper is to study the strong subdifferentiability of the norm of Banach spaces of $N$-homogeneous polynomials $\mathcal{P}(^N X)$ and (its preduals) symmetric tensor products $\sten X$. In order to do so, the following characterization of strong subdifferentiability given in \cite[Theorem 1.2]{FP} (see also \cite{GMZ}) is a very useful tool: the norm $\|\cdot\|$ of a Banach space $X$ is strongly subdifferentiable at a point $x \in S_X$ if and only if for every $\e > 0$, there exists $\delta=\delta(\e, x) > 0$ such that
\begin{equation}\label{characterization FP}
\dist(x^*, D(x)) < \e\quad \text{whenever $x^* \in B_{X^*}$ satisfies $x^*(x) > 1 - \delta$,}
\end{equation}
where $D(x) := \{x^* \in S_{X^*}: x^*(x) = 1\}$. As we will detail in the subsequent sections, this characterization appears in a natural way as \emph{a kind of} Bishop-Phelps-Bollob\'as property for functionals in $X^*$. 
The relation between the denseness of norm attaining mappings and the geometry of the underlying spaces arises in the area from the very beginnings. In his pioneer work \cite{Lin}, Lindenstrauss exhibit some geometrical properties of Banach spaces $X$ and $Y$, that guarantee that the set of norm-attaining linear operators in $\mathcal{L}(X,Y)$ is dense in the whole space. He also showed that the lack of extreme points of the unit ball of the domain space (which is, of course, a geometrical property of the space), plays a fundamental role when trying to obtain examples of spaces for which the set of norm-attaining operators is not dense in the whole space. Since then, this relation between the theory of norm-attaining mappings and the geometry of Banach spaces appears naturally. In particular, it appears in the context of Bishop-Phelps-Bollob\'as type theorems.
Roughly speaking, the Bishop-Phelps-Bollob\'as theorem (the \emph{quantitative} version of Bollob\'as \cite{Bol} of the well-known result of Bishop and Phelps \cite{BP} on the density of norm attaining linear functionals) states that, whenever $(x_0^*, x_0)\in S_{X^*}\times S_{X}$ satisfy $x_0^*(x_0) \approx 1$, there exist $(x_1^*, x_1)\in S_{X^*}\times S_{X}$ such that $x_1^*(x_1)=1$, $x_1^*\approx x_0^*$, and $x_1\approx x_0$. In the context of linear operators, Acosta, Aron, Garc\'ia and Maestre were the first to study, in \cite{AAGM}, a possible generalization of the above mentioned result. As expected, some geometrical properties of the spaces (like the \emph{approximate hyperplane series property} defined and studied in there) appear as sufficient conditions for the validity of a Bishop-Phelps-Bollob\'as theorem for linear operators. For more information on Bishop-Phelps-Bollob\'as type results (in the linear, multilinear and polynomial context) we refer the reader to \cite{A, AAGM, Acoet6, ABGM, ChoKim, CDJ, D, Sain} and the references therein. Recently, some characterizations of uniform convexity, uniform smoothness and strong subdifferentiability of the norm of a Banach space came up to scene in the context of Bishop-Phelps-Bollob\'as type theorems for linear operators and multilinear mappings (see, for instance, \cite{DKL, DKLM, DKLM2, KL}). Having this in mind, we intend to relate the strong subdifferentiability of the norms of $\mathcal{P}(^N X)$ and $\sten X$ with some Bishop-Phelps-Bollob\'as type properties for polynomials (all the proper definitions will be given in Section~\ref{preliminars}). In Theorems~\ref{theoremA} and \ref{multilineartheoremA} we show that, under certain hypotheses on the underlying spaces, the strong subdifferentiability of the space of $N$-homogeneous polynomials $\mathcal{P}(^NX)$ is equivalent to a \emph{polynomial} version of \eqref{characterization FP}. We use these results to study the strong subdifferentiability of the norms of spaces of polynomials and multilinear forms over some classical sequence spaces, such as $\ell_p$, Orlicz spaces and Lorentz spaces. In Section~\ref{Section:TheoremC} we study the \emph{dual} version of this property, and obtain some results on the strong subdifferentiability of elementary tensors on projective symmetric tensor products. As the involved tools and reasonings also apply to the multilinear context, we obtain positive results on the strong subdifferentiability of elementary tensors on the projective tensor product of classical spaces such as $c_0$ and $\ell_2$.

\section{Preliminary material and main results} \label{preliminars} 

In this section, we give the basic concepts we will be using throughout the paper and state our main results. In first place, we set the notation and recall some known properties that will appear in what follows. All Banach spaces considered here are over the real or complex field unless explicitly stated otherwise. We denote, respectively, $B_X$, $S_X$ and $X^*$ the closed unit ball, the unit sphere and the topological dual of a Banach space $X$. We denote by $\mathcal{P}(^N X, Y)$ the Banach space of all $N$-homogeneous polynomials from $X$ into $Y$ endowed with the supremum norm, while $P_{wsc}(^N X, Y)$ stands for the Banach space of all \emph{weakly sequentially continuous} $N$-homogeneous polynomials from $X$ into $Y$. The Banach space of all $N$-linear mappings from $X_1\times \cdots \times X_N$ into $Y$, endowed with the supremum norm, will be denoted by $\mathcal{L}(X_1\times \cdots \times X_N, Y)$, while the space of $N$-linear symmetric mappings from $X\times \cdots\times X$ to $Y$ will be denoted by $\mathcal{L}_s(^NX, Y)$. When $Y=\mathbb{K}$ is the scalar field, we will omit it and write $\mathcal{P}(^NX)$, $\mathcal{P}_{wsc}(^NX)$, $\mathcal{L}(X_1\times \cdots \times X_N)$ and $\mathcal{L}_s(^NX)$. Given $P \in \mathcal{P}(^N X, Y)$ we consider $P^t: Y^* \rightarrow \mathcal{P}(^N X, \K)$, the \emph{transpose of $P$}, given by $(P^t y^*)(x) := y^*P(x)$ for every $x \in X$ and every $y^* \in Y^*$. For $P \in \mathcal{P}(^N X, Y)$, we set $\NA(P) := \{ x \in S_X: \|P(x)\| = \|P\|\}$.
We recall the following well-known properties of Banach spaces. A Banach space $X$ is \emph{strictly convex} (SC, for short) if 
$$
\left\| \frac{x+y}{2}\right\| < 1 \quad \text{whenever $x, y \in S_X$, $x\neq y$}.
$$
The \emph{modulus of convexity} of a Banach space $X$ is defined for each $\e \in (0,2]$ by
\begin{equation*}
\delta_X(\e) :=\inf\left \{ 1-\left\|\frac{x+y}{2}\right\| ~:~ x,y\in B_X, \|x-y\|\geq \e \right\}
\end{equation*}
and $X$ is said to be {\it uniformly convex} (UC, for short) if $\delta_X(\e)>0$ for $\e \in (0,2]$. The \emph{modulus of smoothness} of a Banach space $X$ is defined for each $\tau>0$ by
\begin{equation*}
\rho_X(\tau) :=\sup\left \{\frac{\|x+y\|+\|x-y\|}{2}-1 ~:~ x,y\in X, \|x\|=1, \|y\|=\tau \right\}
\end{equation*}
and $X$ is said to be {\it uniformly smooth} (US, for short) if $\lim_{\tau\to 0} \rho_X(\tau)/\tau=0$. It is a well-known result of $\hat{\text{S}}$mulyan that $X$ is uniformly convex if and only if $X^*$ is uniformly smooth and both properties imply reflexivity. 
We say that a Banach space $X$ has the \emph{Kadec-Klee property} if weak and norm topology coincide on $S_X$. Analogously, $X^*$ has the \emph{$w^*$-Kadec-Klee property} if the weak$^*$ and norm topologies coincide on $S_{X^*}$. We will say that $X$ has the \emph{sequential Kadec-Klee property} is $\|x_n-x_0\|\to 0$ whenever $\|x_n\|\to \|x_0\|$ and $x_n\xrightarrow[]{w}\, x_0$ (analogously, we define the \emph{sequential $w^*$-Kadec-Klee property}). It is worth mentioning that the reader could find a little bit \emph{confusing} the above definition since, in the literature many authors refer to the sequential Kadec-Klee property simply as the Kadec-Klee property.

A Banach space $X$ is said to have the \emph{approximation property} (for short, AP) (respectively, \emph{compact approximation property} (for short, CAP)) if for every compact subset $C$ of $X$ and every $\e>0$, there is a finite rank (respectively, compact) operator $T\colon X\to X$ such that \linebreak $\|Tx - x\|\leq \e$ for every $x \in C$. It is immediate that the AP implies the CAP. However, it is known that there exists a Banach space which has the CAP but does not have the AP \cite{W}.

In the next subsections we focus on some properties and tools that are relevant in what follows.

\subsection{Strong subdifferentiability}\label{section SSD} The following notions (specifically, Definitions \ref{definition:SSD} and \ref{definition:uniform-SSD} below) are the central ones in the present paper. 

\begin{definition} \label{definition:SSD} We say that the norm $\|\cdot\|$ of a Banach space $X$ is \emph{strongly subdifferentiable} (SSD, for short) at $x \in S_X$ when the one-sided limit 
\begin{equation} \label{eq:SSD0}
  \lim_{t \rightarrow 0^+} \frac{\|x + th\| - 1}{t}
\end{equation}
exists uniformly in $h \in B_X$. When it holds for every $x$ in a subset $U\subseteq S_X$ we say that $X$ is SSD on $U$, and when it holds for every $x\in S_X$ we simply say that $X$ is SSD.
\end{definition} 

It is well-known that, for an arbitrary $x \in S_X$, the above limit exists in every direction $h$ and that 
\begin{equation} \label{eq:SSD}
  \tau(x, h) := \lim_{t \rightarrow 0^+} \frac{\|x + th\| - 1}{t} = \max \{ \re x^*(h): x^* \in D(x) \} \ \ \ \ (h \in X),
\end{equation}
where $D(x)$ is the set of all normalized support functionals for $B_X$ at $x$, that is,
\begin{equation*}
  D(x) := \{ x^* \in S_{X^*}: x^*(x) = 1 \}.
\end{equation*}
This means that the norm of $X$ is SSD at $x \in S_X$ if and only if 
\begin{equation} \label{eq:SSD1}
\lim_{t \rightarrow 0^+} \sup \left\{ \frac{\|x + th\| - 1}{t} - \tau(x, h): h \in B_X \right\} = 0.  
\end{equation}
As we already mentioned in the Introduction, Franchetti and Pay\'a \cite{FP} proved the following characterization result. 
\begin{theorem}[\mbox{\cite[Theorem 1.2]{FP}}]\label{thm:FP}
Let $X$ be a Banach space. The norm of $X$ is SSD at $x\in S_X$ if
and only if given $\eps >0$, there exists $\delta(\eps,x)>0$ such that whenever $x^* (x) > 1-\delta(\eps,x)$ for some $x^* \in B_{X^*}$, the distance $d(x^*, D(X)) < \eps$.
\end{theorem}
In other words, the norm of $X$ is SSD at $x\in S_X$ if and only if $x$ \emph{strongly exposes} the set $D(x)$ (that is, the distance $\dist(x_n^*, D(x))$ tends to zero for any sequence $(x_n^*)_{n=1}^{\infty} \subseteq B_{X^*}$ with $\re x_n^*(x) \rightarrow 1$ as $n \rightarrow \infty$).
It is worth mentioning that this is the analogue of the characterization for Fr\'echet differentiability proved by \v{S}mulyan in \cite{Smu}.

For a background on the study of strong subdifferentiability of the norm, we send the reader to \cite{AOPR, Contreras, CP, F, GGS, Godefroy, GMZ, Gregory}. For a systematic study on the topic, we suggest \cite{FP, DGZ}. Here, we shall only mention some examples of classical Banach spaces with SSD norm: it is known that every finite-dimensional Banach space is SSD. Since a Banach space $X$ is uniformly smooth if and only if its norm is uniformly Fr\'echet differentiable on $S_X$, it is clear that every uniformly smooth Banach space is SSD; for instance, the sequence spaces $\ell_p$ with $1<p<\infty$ are SSD. It is worth mentioning that if a dual space $X^*$ is SSD, then $X$ is reflexive (see \cite[Theorem~ 3.3]{FP}); hence $\ell_1$ and $\ell_\infty$ are not SSD. The sequence spaces $c_0$ and $d_*(w,1)$ (predual of Lorentz sequence space) are examples of non-reflexive SSD Banach spaces. There are other examples of non-reflexive spaces with a SSD norm, for instance, the predual of the Hardy space $H^1$ and the predual of the Lorentz space $L_{p,1}(\mu)$. Moreover, if $X$ is a predual of a Banach space with the $w^*$-Kadec-Klee property, then $X$ is SSD (see \cite[Proposition~2.6]{DKLM}). On the other hand, it is known that if $X$ is SSD, then $X$ is an Asplund space (see \cite{FP, GMZ}).

We now deal with a uniform and, at the same time, localized version of strong subdifferentiability. It was proved in \cite[Proposition~4.1]{FP} that a Banach space $X$ is uniformly smooth if and only if the limit in \eqref{eq:SSD0} is also uniform in $x\in S_X$ (we already mentioned in the above paragraph, that if $X$ is uniformly smooth, then $X$ is SSD). Since uniform smoothness is a quite restrictive property, we could ask for Banach spaces for which the limit in \eqref{eq:SSD0} is uniform in $x\in U$, for some subset $U\subset S_X$. 
\begin{definition} \label{definition:uniform-SSD} Given a set $U \subseteq S_X$, we say that the norm of a Banach space $X$ is \emph{uniformly strongly subdifferentiable on $U$} (USSD on $U$, for short) if the limit (\ref{eq:SSD0}) is uniform for $h \in B_X$ and $x \in U$. In other words, the norm of $X$ is USSD on $U$ if and only if 
\begin{equation} \label{eq:SSD3}
\lim_{t \rightarrow 0^+} \sup \left\{ \frac{\|x+th\| - 1}{t} - \tau(x, h): h \in B_X, x \in U \right\} = 0.
\end{equation}
\end{definition}
A relation between numerical range and uniform strong subdifferentiability was established in \cite{Rod}. Also,
the uniform strong subdifferentiability of the norm of JB$^*$-triples was studied in \cite{BecRod}, where it is proved that if $X$ is a JB$^*$-triple, then $X$ is USSD on the set of nonzero tripotents of $X$. However, we could not find a systematic study of this property in the literature and, hence, we are not able to give more examples of Banach spaces $X$ and subsets $U\subset S_X$ on which $X$ is USSD.

\subsection{(Symmetric) tensor products} On the one hand, the \emph{projective tensor product} between the Banach spaces $X_1, \dots, X_N$, denoted by $X_1 \pten \cdots \pten X_N$, is defined as the completion of the algebraic tensor product $X_1 \otimes \cdots \otimes X_N$ endowed with the norm 
\begin{equation*}
\|z\|_{\pi} := \inf \left\{ \sum_{i=1}^n \|x^1_i\|\cdots \|x^N_i\|: z = \sum_{i=1}^{n} x^1_i \otimes\cdots \otimes x^N_i \right\},
\end{equation*}
where the infimum is taken over all representations of $z$ of the form $\sum_{i=1}^n x^1_i \otimes\cdots \otimes x^N_i$. It is well-known that the tensor product between $X_1, \dots, X_N$ \emph{linearizes} $N$-linear mappings on $X_1\times\cdots \times X_N$. Indeed, we have the isometric isomorphism
$$
(X_1 \pten \cdots \pten X_N)^*=\mathcal{L}(X_1\times \cdots\times X_N),
$$
where the duality is given by 
$$
L_A(z)=\langle z, A\rangle=\sum_{i=1}^\infty A(x^1_i, \dots, x^N_i),
$$
for $A\in \mathcal{L}(X_1\times \cdots\times X_N)$ and $z= \sum_{i=1}^\infty x^1_i \otimes\cdots \otimes x^N_i \in X_1 \pten \cdots \pten X_N$. Moreover, for multilinear mappings with values in a dual space $Y^*$ we have the isometric isomorphism
$$
\left((X_1 \pten \cdots \pten X_N)\pten Y\right)^*=\mathcal{L}(X_1\times \cdots\times X_N, Y^*),
$$
with the duality given by
$$
L_A(z)=\langle z, A\rangle=\sum_{j=1}^\infty\sum_{i=1}^\infty A(x^1_{j,i}, \dots, x^N_{j,i})(y_j),
$$
for $A\in \mathcal{L}(X_1\times \cdots\times X_N, Y^*)$ and $z= \sum_{j=1}^\infty v_j\otimes y_j$, where $(y_j)_j \subset Y$ and $(v_j)_j \subset X_1\pten\cdots\pten X_N$ with $v_j=\sum_{i=1}^\infty x^1_{j,i}\otimes \cdots\otimes x^N_{j,i}$.
It is well-known that the closed unit ball of $X_1\pten X_2$ is the closed convex hull of $B_{X_1}\otimes B_{X_2}$, that is, $B_{X_1 \pten X_2} = \overline{\co}(B_{X_1} \otimes B_{X_2})$.

On the other hand, the \emph{symmetric projective tensor product} of $X$, denoted by $\sten X$, is the completion of the linear space $\otimes_{s,N} X$ generated by $\{ \otimes^N x: x \in X\}$ (here, $\otimes^N x$ stands for the elementary tensor $x \otimes \stackrel{N}{\cdots} \otimes x$) endowed with the norm 
\begin{equation*}
\|z\|_{\pi_s, N} := \inf \left\{ \sum_{i=1}^n |\lambda_i|\|x_i\|^N: z = \sum_{i=1}^n \lambda_i \otimes^N x_i \right\}, 
\end{equation*}
where the infimum is taken over all the possible representations of $z$ of that form. In the same way that projective tensor product linearizes multilinear mappings, the symmetric projective tensor product linearizes homogeneous polynomials. The identity
$$
(\sten X)^* = \mathcal{P}(^N X)
$$ 
holds isometrically, and the duality is given by
\begin{equation*}
L_P(z) =\langle z, P \rangle=\sum_{i=1}^\infty \lambda_i P(x_i)
\end{equation*}
for $P\in \mathcal{P}(^NX)$ and $z=\sum_{i=1}^\infty \lambda_i \otimes^N x_i \in \sten X$. More in general, 
$$
((\sten X) \pten Y)^* = \mathcal{P}(^N X, Y^*)
$$
holds isometrically, with the duality given by
$$
L_P(z) =\langle z, P \rangle=\sum_{j=1}^\infty\sum_{i=1}^\infty \lambda_{j,i}P(x_{j,i})(y_j)
$$
for $P\in \mathcal{P}(^N X, Y^*)$ and $z=\sum_{j=1}^\infty v_j\otimes y_j$ for $(y_j)_j \subset Y$ and $(v_j)_j\subset \sten X$ with \linebreak $v_j=\sum_{i=1}^\infty \lambda_{j,i}\otimes^N x_{j,i}$.
We also have that $B_{\sten X} = \overline{\aco}(\{ \otimes^N x: x \in S_X \})$ where $\aco(C)$ stands for the absolute convex hull of the set $C$. 

We refer the reader to the first chapters of the books \cite{defant1992tensor, ryan2002introduction} for an introduction on tensor products (see also \cite{diestel2008metric}), and to Floret's survey article \cite{floret1997natural} for symmetric tensor products.

\subsection{Orlicz and Lorentz sequence spaces}\label{Orlicz Lorentz}
We briefly recall the definitions and some properties of Orlicz and Lorentz sequence spaces. These spaces, as well as $\ell_p$ spaces, will come up as examples of applications of our main results. An \emph{Orlicz function} $M$ is a continuous non-decreasing and convex function defined for $t\geq 0$ such that $M(0)=0$, $M(t)>0$ for every $t>0$ and $\lim_{t\to\infty}M(t)=\infty$. The \emph{Orlicz sequence space $l_M$} associated to an Orlicz function $M$ is the space of all sequences of scalars $x=(a_i)_i$ such that $\sum_{i}M(|a_i|/\rho)<\infty$ for some $\rho>0$. The space $l_M$ equipped with the Luxemburg norm
$$
\|x\|=\inf\left\{ \rho>0:\,\, \sum_{i=1}^\infty M(|a_i|/\rho)\leq 1\right\}
$$
is a Banach space, and we are dealing with this norm unless stated the contrary. An Orlicz function $M$ is said to satisfy the \emph{$\Delta_2$-condition at zero} if
$$
\limsup_{t\to 0} \frac{M(2t)}{M(t)}<\infty.
$$
The canonical vectors $\{e_n\}_n$ form a symmetric basic sequence in $l_M$ and a symmetric basis of the subspace $h_M\subset l_M$ consisting of those sequences $x=(a_i)_i\in l_M$ such that\linebreak $\sum_{i} M(|a_i|/\rho)<\infty$ for every $\rho>0$. The equality $l_M=h_M$ holds if and only if $M$ satisfies the $\Delta_2$-condition at zero. In this case it is clear that $l_M$ has the CAP, since it has a Schauder basis. Let
\begin{equation}\label{Boyd index alpha}
\alpha_M=\sup\left\{p>0:\,\, \sup_{0<t,\lambda\leq 1} \frac{M(\lambda t)}{M(\lambda)t^p}<\infty \right\}
\end{equation}
and
\begin{equation}\label{Boyd index beta}
\beta_M=\sup\left\{q>0:\,\, \inf_{0<t,\lambda\leq 1} \frac{M(\lambda t)}{M(\lambda)t^q}>0 \right\}.
\end{equation}
It is known that $1\leq \alpha_M \leq \beta_M\leq \infty$, and $\beta_M<\infty$ if and only if $M$ satisfies the $\Delta_2$-condition at zero. Moreover, the space $l_M$ is reflexive if and only if $\beta_M<\infty$ and $\alpha_M>1$ or, equivalently, $M$ and its dual function $M^*(u)=\max\{tu - M(t):\,\, 0<t<\infty\}$ satisfy the $\Delta_2$-condition at zero. It is also known that $l_M$ has the (uniform) Kadec-Klee property if and only if $M$ satisfies the \linebreak $\Delta_2$-condition at zero.
A detailed study of these and other properties of Orlicz sequence spaces can be found, for instance, in \cite{LT1}. 
Another properties in which we are particularly interested are uniform convexity and uniform smoothness. In \cite[Theorem~2.38]{Che} it is shown that the space $l_M$ endowed with the Orlicz norm
$$
\|x\|^0=\sup\left\{\sum_{i=1}^\infty a_i b_i:\,\, \sum_{i=1}^\infty M(|b_i|)\leq 1\right\}
$$
is uniformly convex if and only if $M$ satisfies the $\Delta_2$-condition at zero and $M$ is uniformly convex on $[0, \pi_M(1)]$, where $\pi_M(\alpha)=\inf\{t>0:\,\, M^*(p(t))\geq \alpha\}$ (here, $p$ is the right derivative of $M$),
i.e., given $\e>0$ there exists $\delta>0$ such that
$$
M\left(\frac{t+s}{2}\right)\leq (1-\delta)\frac{M(t)+M(s)}{2}
$$
for all $s,t\in [0, \pi_M(1)]$ satisfying $|s-t|\geq \e \max\{s,t\}$. Since $(l_M, \|\cdot\|)=(l_{M^*}, \|\cdot\|^0)^*$ isometrically, we obtain necessary and sufficient conditions for the uniform smoothness of $l_M$ (with the Luxemburg norm).

We focus now our attention on Lorentz sequence spaces. Let $1\leq p<\infty$ and $v=(v_i)_i$ be a non-increasing sequence of positive numbers such that $v_1=1$, $\lim_i v_i=0$ and $\sum_i v_i=\infty$. The \emph{Lorentz sequence space $d(v,p)$} is the Banach space of all sequences $x=(a_i)_i$ such that 
$$
\|x\|=\sup_{\pi} \left( \sum_{i=1}^\infty v_i \,|a_{\pi(i)}|^p\right)^\frac{1}{p} <\infty,
$$
where the supremum is taken over all permutations $\pi$ of the set of positive integers. It is well-known that $d(v,p)$ is reflexive if and only if $1<p<\infty$.
The canonical vectors $\{e_n\}_n$ form a symmetric basic sequence in $d(v,p)$ and, consequently, has the CAP.
In \cite[Theorem~2]{CasLin} it is proved that $d(v,p)$ has the sequential Kadec-Klee property if $1<p<\infty$. Moreover, in \cite{Alt} it is shown that $d(v,p)$ ($1<p<\infty$) is uniformly convex if and only if 
$$
\inf_n \frac{\sum_{i=1}^{2n} v_i}{\sum_{i=1}^{n} v_i}=k>1.
$$
For basic properties of Lorentz sequence spaces we refer the reader to \cite{LT1}.

As we will see below, Theorem~\ref{theoremA} relates the strong subdifferentiability of the space of homogeneous polynomials with the study of weakly sequentially continuous polynomials. In that sense, the lower and upper indexes of a Banach space $X$ defined by Gonzalo and Jaramillo in \cite{GJ} will appear naturally in our context, since they are closely related to the study of weakly sequentially continuous polynomials. We are particularly interested in the values of these indexes for Orlicz and Lorentz sequence spaces, computed by Gonzalo in \cite{Gon}. In first place, we recall the definition of lower and upper indexes of a Banach space $X$. A sequence $(x_n)_n$ in $X$ is said to have an \emph{upper $p$-estimate} ($1\leq p\leq \infty$) if there exist a constant $C$ such that
$$
\left\| \sum_{n=1}^n a_n x_n\right\| \leq C \left(\sum_{n=1}^n |a_n|^p \right)^{\frac{1}{p}}
$$
for every $n$-tuple of scalars $a_1, \dots, a_n$. A Banach space $X$ has \emph{property $S_p$} if every weakly null semi-normalized basic sequence in $X$ has a subsequence with an upper $p$-estimate. The \emph{lower index of $X$} is defined as
$$
l(X)=\sup\{p\geq 1:\,\, \text{$X$ has property $S_p$}\}.
$$
Analogously, using lower $q$-estimates (instead of upper $p$-estimates) it can be defined the \emph{property $T_q$} and the \emph{upper index of $X$} as
$$
u(X)=\inf\{q\geq 1:\,\, \text{$X$ has property $T_q$}\}.
$$
It is not difficult to see that $l(\ell_p)=u(\ell_p)=p$ for $1<p<\infty$. As we already mentioned, in \cite{Gon} the author computes the values of lower and upper indexes for Orlicz and Lorentz sequence spaces. In the case of Orlicz spaces, it is shown that $l(h_M)=\alpha_M$ and $u(h_M)=\beta_M$, where $\alpha_M, \beta_M$ are the lower and upper Boyd indexes defined in \eqref{Boyd index alpha} and \eqref{Boyd index beta}. Since we are interested in reflexive Orlicz spaces and, in that case, the equality $l_M=h_M$ holds, we will use that $l(l_M)=\alpha_M$ and $u(l_M)=\beta_M$. For Lorentz sequence spaces it is not known (up to our knowledge) the exact values of both lower and upper indexes. On the one hand, it is known that $l(d(v,p))=p$ for $1<p<\infty$. On the other hand $u(d(v,p))\geq r^*(v)p$, where 
$$
r(v)=\inf\{s\in [1,\infty]:\,\, v\in \ell_s\} \quad \text{and}\quad \frac{1}{r(v)}+\frac{1}{r^*(v)}=1.
$$

\subsection{Motivation and tools} Recall that the Bishop-Phelps-Bollob\'as theorem states that given $\e > 0$, there exists $\eta(\e) > 0$ such that whenever $(x_0^*, x_0)\in S_{X^*}\times S_{X}$ satisfy $|x_0^*(x_0)| > 1 - \eta(\e)$, there exist $(x_1^*, x_1)\in S_{X^*}\times S_{X}$ such that 
$$
|x_1^*(x_1)|=1, \quad \|x_1^*-x_0^*\|<\e \quad \text{and}\quad \|x_1-x_0\|<\e.
$$
Note that the characterization of strong subdifferentiability stated in Theroem \ref{thm:FP} is \emph{a kind of} Bishop-Phelps-Bollob\'as property on which the point $x_0$ is fixed and the $\eta$ in the definition depends not only on $\e>0$ but also on the fixed point $x_0$ (in that sense, these property is referred as the \emph{local Bishop-Phelps-Bollob\'as point property} since we \emph{fix} a point and the $\eta$ is \emph{localized}). 
As it is natural in the study of Bishop-Phelps-Bollob\'as type properties, the above property and its \emph{dual} counterpart (where, instead of a point, a linear functional is fixed and the function $\eta$ depends on $\e$ and the fixed functional) were defined and studied in the context of linear and multilinear operators. It is worth mentioning that these properties are related to the strong subdifferentiability of the domain spaces and, in the multilinear context, of the projective tensor product of the domain spaces. For the sake of clarity, let us define the local Bishop-Phelps-Bollob\'as properties. We define these properties in the more general context of $N$-linear mappings, which covers the \emph{linear case} (putting $N=1$) and also the \emph{functional case} (putting $N=1$ and the scalar field as the range space). As we are going to deal with the strong subdifferentiability of spaces of polynomials and symmetric projective tensor products, we also define the polynomial versions of such properties. We follow the notation in \cite{DKLM, DKLM2, DR}.
\begin{definition}\label{def Lpp and Loo}
Let $N\in \mathbb{N}$ and $X, X_1, \dots, X_N, Y$ be Banach spaces.
\begin{enumerate}
  \item[(i)] The \emph{local Bishop-Phelps-Bollob\'as point property} ($\Lpp$, for short).
  
\medskip  

\noindent The pair $(X_1 \times \cdots \times X_N, Y)$ has the $\Lpp$ if given $\e > 0$ and $(x_1,\dots, x_N) \in S_{X_1} \times \cdots \times S_{X_N}$, there exists $\eta(\e, x_1, \ldots, x_N) >0$ such that whenever $A \in \mathcal{L}(X_1 \times \cdots \times X_N, Y)$ with $\|A\| = 1$ satisfies $$\|A(x_1, \ldots, x_N)\| > 1 - \eta(\e, x_1, \ldots, x_N),$$ there exists $B \in \mathcal{L}(X_1 \times \cdots \times X_N, Y)$ with $\|B\| = 1$ such that 
  $$
  \|B(x_1, \ldots, x_N)\| = 1 \quad \text{and}\quad \|B - A\| <\e.
  $$
  
\medskip
  
\noindent The pair $(X, Y)$ has the \emph{$N$-homogeneous polynomial $\Lpp$} if given $\e > 0$ and $x\in S_X$, there exists $\eta(\e, x) > 0$ such that whenever $P\in S_{\mathcal{P}(^NX, Y)}$ satisfy $\|P(x)\|>1-\eta(\e, x)$, there exists $Q\in \mathcal{P}(^NX, Y)$ such that $\|Q(x)\|=1$ and $\|P-Q\|<\e$. 

\medskip

\item[(ii)] The \emph{local Bishop-Phelps-Bollob\'as operator property} ($\Loo$, for short).
  
\medskip 

\noindent The pair $(X_1 \times\cdots\times X_N, Y)$ has the $\Loo$ if given $\e > 0$ and $A \in \mathcal{L}(X_1\times\cdots \times X_N, Y)$ with $\|A\| = 1$, then there exists $\eta(\e, A) > 0$ such that whenever $(x_1,\ldots,x_N) \in S_{X_1}\times\cdots \times S_{X_N}$ satisfies 
$$
\|A(x_1,\ldots,x_N)\| > 1 - \eta(\e, A),
$$
there exists $(x_1^0,\ldots,x_N^0) \in S_{X_1}\times\cdots \times S_{X_N}$ such that
$$
\|A(x_1^0,\ldots,x_N^0)\| = 1 \quad \text{and} \quad \|x^0_i - x_i\| < \e
$$
for every $i=1, \ldots, N$. 

\medskip

\noindent The pair $(X, Y)$ has the \emph{$N$-homogeneous polynomial $\Loo$} if given $\e > 0$ and \linebreak $P \in \mathcal{P}(^NX, Y)$ with $\|P\| = 1$, there exists $\eta(\e, P) > 0$ such that whenever $x \in S_X$ satisfies $\|P(x)\| > 1 - \eta(\e, P)$, there exists $x_0 \in S_X$ such that $\|P(x_0)\| = 1$ and $\|x_0 - x\| < \e$.
\end{enumerate}
\end{definition}

Let us briefly explain the connection between these Bishop-Phelps-Bollob\'as type properties and the geometry of the underlying Banach spaces. In first place, as an easy consequence of the characterizations of strong subdifferentiability given in \cite[Theorem~1.2]{FP} we have that:
\begin{itemize}
\itemsep0.3em
  \item $X$ is SSD if and only if the pair $(X, \mathbb{K})$ has the $\Lpp$;
  \item $X^*$ is SSD if and only if the pair $(X, \mathbb{K})$ has the $\Loo$.
\end{itemize}
When dealing with vector-valued linear and multilinear operators, we have only one of the implications in the above equivalences.
\begin{itemize}
\itemsep0.3em
  \item If the pair $(X_1\times \cdots\times X_N, Y)$ has the $\Lpp$, then $X_i$ is SSD for every $i=1,\dots, N$ (see \linebreak \cite[Proposition~2.3]{DKLM2}). The reciprocal does not hold (see \cite[Remark~3.3]{DKLM}).
  \item If the pair $(X_1\times \cdots\times X_N, Y)$ has the $\Loo$, then $X_i^*$ is SSD for every $i=1,\dots, N$\linebreak (see \cite[Proposition~2.3]{DKLM2}). The reciprocal does not hold (see \cite[Theorem~2.1]{D}).
\end{itemize}
At this point, we are ready to point out our major motivation in the study of strong subdifferentiability of the spaces of $N$-homogeneous polynomials and symmetric tensor products. Since the projective tensor product of two Banach spaces $X_1$ and $X_2$ \emph{linearizes} the space of bilinear forms on $X_1\times X_2$, the following questions come up naturally: 
\begin{itemize}
\itemsep0.3em
\item[Q1)] Does the pair $(X_1\times X_2, \mathbb{K})$ has the $\Lpp$ if and only if the pair $(X_1\pten X_2, \mathbb{K})$ has the $\Lpp$? Note that this last statement is equivalent to say that $X_1\pten X_2$ is SSD.
\item[Q2)] Does the pair $(X_1\times X_2, \mathbb{K})$ has the $\Loo$ if and only if the pair $(X_1\pten X_2, \mathbb{K})$ has the $\Loo$? Note that this last statement is equivalent to say that $(X_1\pten X_2)^*$ is SSD.
\end{itemize}
These questions were addressed in \cite{DKLM2, DR} and (among others) the following results were obtained.
\begin{theorem}\label{thm motivation}
Let $X_1, X_2$ be Banach spaces.
\begin{enumerate}
\itemsep0.3em 
  \item[\rm (i)] If $X_1\pten X_2$ is SSD then the pair $(X_1\times X_2, \mathbb{K})$ has the $\Lpp$. The reciprocal does not hold taking, for instance, $X_1=X_2=\ell_2$.
  \item[\rm (ii)] Suppose that $X_1$ has the AP. If $X_1$ is strictly convex or has the sequential Kadec-Klee property, then $(X_1\times X_2, \mathbb{K})$ has the $\Loo$ if and only if the pair $(X_1\pten X_2, \mathbb{K})$ has the $\Loo$.
\end{enumerate}
\end{theorem}
Our goal is to obtain \emph{differentiability properties} of symmetric tensor products and its dual spaces (the spaces of homogeneous polynomials). In that sense, a \emph{polynomial version} of Theorem~\ref{thm motivation}(ii) would be helpful to obtain a relation between strong subdifferentiability of $(\sten X)^*=\mathcal{P}(^NX)$ and the $\NLoo$. Such a relation would help us to prove the strong subdifferentiability of $\mathcal{P}(^NX)$ for many Banach spaces $X$. A similar argument could be reproduced to obtain strong subdifferentiability of the symmetric tensor product $\sten X$, if we could find a relation between this property and the $N$-homogeneous polynomial ${\bf L}_{p,p}$. Unfortunately (or not) we cannot expect that since, as we state in Theorem~\ref{thm motivation}(i), even in the bilinear case the strong subdifferentiability of the projective tensor product cannot be deduced from the \emph{bilinear} $\Lpp$ property. At this point is where the (uniform) strong subdifferentiability on certain subsets appears, as the link between a differentiability property of the symmetric tensor product and the (uniform) $N$-homogeneous polynomial ${\bf L}_{p,p}$.

\subsection{Main results} We state now our main results, which will be proved in Sections~\ref{Section:TheoremAandB} and \ref{Section:TheoremC}. Although we are mainly interested in differentiability properties of spaces of polynomials and symmetric tensor products, we also state some results in the context of multilinear mappings and (full, not symmetric) projective tensor products. We focus first in the relation between strong subdifferentiability of the space of homogeneous polynomials and the $\NLoo$ (see Definition~\ref{def Lpp and Loo}). It is worth noting that we obtain results not only in the scalar-valued case, but also for polynomials taking values in a uniformly smooth space. 

\begin{mainth} \label{theoremA} 
Let $N \in \N$, let $X$ be a (reflexive) Banach space with the CAP and the sequential Kadec-Klee property and let $Y$ be a uniformly convex Banach space. Then, the following are equivalent.
	\begin{enumerate}
	\itemsep0.3em 
		\item[(a)] $\mathcal{P}(^N X, Y^*)$ is SSD.
		\item[(b)] The pair $\left( \left( \sten X \right) \pten Y, \K \right)$ has the {\bf L}$_{o,o}$ (for linear functionals).	
		\item[(c)] $\mathcal{P}(^N X, Y^*)$ is reflexive.
		\item[(d)] $\mathcal{P}(^N X, Y^*)=\mathcal{P}_{wsc}(^N X, Y^*)$.
		\item[(e)] The pair $(X, Y^*)$ has the $N$-homogeneous polynomial $\Loo$.
	\end{enumerate} 
\end{mainth}

As a consequence of the previous equivalence we deduce strong subdifferentiability of spaces of homogeneous polynomials between $\ell_p$, Lorentz and Orlicz sequence spaces. In view of Theorem~\ref{theoremA}, the involved spaces should satisfy certain hypotheses such as the compact approximation property, the sequential Kadec-Klee property and uniform convexity. Necessary and sufficient conditions for reflexivity, uniform convexity (and uniform smoothness) of Orlicz and Lorentz sequence spaces, as well as Boyd indexes $\alpha_M$ and $\beta_M$, were stated in subsection~\ref{Orlicz Lorentz}. It was also noted in there that the involved spaces satisfy the CAP and the sequential Kadec-Klee property. 

\begin{maincor}\label{corollaryA}
Let $1<p,q<\infty$ and let $M_1,M_2$ be Orlicz functions such that $1<\alpha_{M_i}, \beta_{M_i}<\infty$ for $i=1,2$. Suppose that $l_{M_2}$ is uniformly smooth.
\begin{enumerate}
\itemsep0.3em 
\item[\rm (i)] $\mathcal{P}(^N \ell_p)$ is SSD if and only if $N < p$. 
\item[\rm (ii)] $\mathcal{P}(^N \ell_p, \ell_{q})$ is SSD if and only if $Nq<p$. 
\item[\rm (iii)] $\mathcal{P}(^N l_{M_1})$ is SSD if and only if $N < \alpha_{M_1}$.
\item[\rm (iv)] $\mathcal{P}(^N l_{M_1}, l_{M_2})$ is SSD if and only if $N \beta_{M_2} < \alpha_{M_1}$.
\item[\rm (v)] $\mathcal{P}(^N d(w,p))$ is SSD if and only if $N < p$.
\item[\rm (vi)] $\mathcal{P}(^N d(w,p), l_{M_2})$ is SSD if and only if $N\beta_{M_2} < p$. 
\end{enumerate}
\end{maincor}

Note that the equivalence between (a) and (b) in Theorem~\ref{theoremA} follows immediately from the isometry $\mathcal{P}(^N X, Y^*)=(\left( \sten X \right) \pten Y)^*$ and the fact that the pair $(X, \mathbb{K})$ has the $\Loo$ if and only if $X^*$ is SSD. The implication (b)$\Rightarrow$(c) is trivial since, as we already mentioned in subsection~\ref{definition:SSD}, strongly subdifferentiable dual spaces are reflexive. The equivalence (c)$\Leftrightarrow$(d) is essentially contained in \cite{Muj} (see also \cite{J}), where it is proved that, if $X, Y$ are reflexive and $X$ has the CAP, then $\mathcal{P}(^NX,Y)$ is reflexive if and only if $\mathcal{P}(^NX,Y)=\mathcal{P}_w(^NX,Y)$, the subspace of polynomials which are weakly continuous on bounded sets. Since $X$ does not contain a copy of $\ell_1$, we have $\mathcal{P}_w(^NX,Y)=\mathcal{P}_{wsc}(^NX,Y)$ (see \cite[Theorem~2.9 and Proposition~2.12]{AroHerVal}) and this completes the equivalence between (c) and (d). The implications (d)$\Rightarrow$(e)$\Rightarrow$(a) are left for the next section; in the first one we need the sequential Kadec-Klee property of the space $X$, while the second holds for every reflexive Banach space $X$ and every uniformly convex Banach space $Y$, with no extra assumptions on $X$.

\begin{remark}\label{remark:Loo} Let us make some relevant observations related to the previous theorem.
\begin{enumerate}  
\item[(i)] In view of Theorem~\ref{theoremA} and Corollary~\ref{corollaryA}, it is natural to ask if there exist some reflexive Banach space which is not SSD. In \cite[Example~2]{Smith} it is shown an example of a reflexive Banach space $Z$ (a renorming of $\ell_2$) which is strictly convex but is not \emph{midpoint locally uniformly rotund} (see the proper definition in the mentioned article). Then $Z^*$ is reflexive and is not SSD. Indeed, by \cite[Theorem~2.5]{DKLM}, if a dual space $X^*$ is SSD, then $X$ is strictly convex if and only if $X$ is midpoint locally uniformly rotund. This example shows that the strong subdifferentiability is, indeed, a stronger property than reflexivity for dual spaces.
\item[\rm (ii)] Note that any of the statements in the previous theorem imply that the space $X$ is reflexive and, hence, this is a vaquous hypothesis (for this reason it appears in parenthesis). Indeed, if $\mathcal{P}(^N X, Y^*)=(\left( \sten X \right) \pten Y)^*$ is SSD, then it is reflexive (we already mentioned in Subsection~\ref{definition:SSD} that strongly subdifferentiable dual spaces are reflexive) and, consequently, $X$ and $Y$ are reflexive. Also, the $N$-homogeneous polynomial $\Loo$ of the pair $(X, Y^*)$ implies that every $P\in \mathcal{P}(^N X,Y^*)$ is norm-attaining and, consequently, every functional in $X^*$ attains its norm. Then, by James' theorem, $X$ is reflexive. It is worth mentioning that the reflexivity of the space $Y$ is necessary in (e)$\Rightarrow$(a). Indeed, if $X$ is a finite-dimensional space then the pair $(X,Y)$ has the $N$-homogeneous polynomial $\Loo$ for every Banach space $Y$ (the proof is analogous to \cite[Theorem~2.4]{D}, where the statement is proved in the linear setting). Then, if $X$ is finite-dimensional and $Y$ is non-reflexive, the pair $(X,Y)$ has the $N$-homogeneous polynomial $\Loo$ and $\mathcal{P}(^NX,Y)$ is not SSD, since fails to be reflexive. Moreover, in order to obtain examples of spaces $X, Y$ such that $\mathcal{P}(^NX,Y^*)$ is SSD, $X^*$ and $Y^*$ need to be SSD. Indeed, it is not difficult to see that if $(\left( \sten X \right) \pten Y)^*$ is SSD, then $X^*$ and $Y^*$ are SSD (this can be found, for instance, in the proof of \cite[Theorem~B]{DR}).
\item[\rm (iii)] In view of the previous remark, it is natural to ask if, in Theorem~\ref{theoremA}, the uniform convexity of $Y$ can be relaxed to "$Y^*$ is SSD". We do not know if we can change the uniform convexity hypothesis, which we use in the proof of (e)$\Rightarrow$(a).
\item[\rm (iv)] If $Y$ is reflexive and every $P\in \mathcal{P}(^NX,Y^*)$ attains its norm, then $\mathcal{P}(^NX,Y^*)$ is reflexive. Indeed, in virtue of the isometry $\mathcal{P}(^N X, Y^*)=((\sten X) \pten Y)^*$ and the hypotheses (here we need the reflexivity of $Y$), we deduce that every linear functional in $((\sten X) \pten Y)^*$ is norm attaining and, hence, $((\sten X) \pten Y)^*$ is reflexive (this is proved, for instance, in \cite[Theorem~1.3]{JarMor}). 
\item[\rm (v)] A linear bounded operator defined on a reflexive Banach space is compact if and only if maps weakly convergent sequences into norm convergent sequences. Having this in mind, it might be worth mentioning that item (d) in Theorem~\ref{theoremA} is not equivalent to say that every $P \in \mathcal{P}(^N X, Y^*)$ is a \emph{compact polynomial}, provided that $N\geq 2$ (recall that $P \in \mathcal{P}(^N X, Y^*)$ is compact if maps the unit ball of $X$ into a relatively compact set of $Y^*$). For instance, $N$-homogeneous polynomials from $\ell_2$ to $\mathbb{K}$ are compact but $\mathcal{P}(^N \ell_2)$ is not reflexive provided $N\geq 2$. It is also worth mentioning that, if $X$ and $Y^*$ are reflexive Banach spaces both with the approximation property, then item (d) (or, equivalently, item (c)) is equivalent to say that every $P \in \mathcal{P}(^N X, Y^*)$ is approximable, that is, is in the closure of finite type polynomials. This was proved by Alencar in \cite{Ale}.
\end{enumerate}
\end{remark}

We state now the multilinear counterpart of Theorem~\ref{theoremA} and Corollary~\ref{corollaryA}.

\begin{mainth} \label{multilineartheoremA} 
Let $N \in \N$ and $X_1,\ldots,X_N$ be reflexive Banach spaces with Schauder bases such that $X_1, \dots, X_{N-1}$ have the sequential Kadec-Klee property and $X_N$ is uniformly convex. Then, the following are equivalent.
\begin{enumerate}
\itemsep0.3em 
\item[(a)] The norm of $\mathcal{L}(X_1\times\cdots\times X_N)=\mathcal{L}(X_1\times\cdots\times X_{N-1},X_N^*)$ is SSD.
\item[(b)] The pair $\left( X_1\pten\cdots \pten X_N , \K \right)$ has the {\bf L}$_{o,o}$ (for linear functionals).
\item[(c)] $\mathcal{L}(X_1\times\cdots\times X_N)$ is reflexive.
\item[(d)] $\mathcal{L}(X_1\times\cdots\times X_{N-1}, X_N^*)=\mathcal{L}_{wsc}(X_1\times\cdots\times X_{N-1}, X_N^*)$.
\item[(e)] The pair $(X_1\times \cdots\times X_N, \K)$ has the $\Loo$ (for multilinear forms).
\end{enumerate}
\end{mainth}


\begin{maincor}\label{corollaryB}
Let $1<p_1, \dots, p_N,q<\infty$ and let $M_1,\dots, M_{N+1}$ be Orlicz functions satisfying the $\Delta_2$-condition and such that $1<\alpha_{M_1}, \beta_{M_1}, \dots, \alpha_{M_{N+1}}, \beta_{M_{N+1}}<\infty$. Suppose also that $l_{M_{N+1}}$ is uniformly smooth.
\begin{enumerate}
\itemsep0.3em 
\item[\rm (i)] $\mathcal{L}(\ell_{p_1}\times \cdots\times \ell_{p_N})$ is SSD if and only if $\frac{1}{p_1}+\cdots+\frac{1}{p_N}<1$. 
\item[\rm (ii)] $\mathcal{L}(\ell_{p_1}\times \cdots\times \ell_{p_N}, \ell_q)$ is SSD if and only if $\frac{1}{p_1}+\cdots+\frac{1}{p_N}<\frac{1}{q}$.
\item[\rm (iii)] $\mathcal{L}(l_{M_1}\times \cdots\times l_{M_N})$ is SSD if and only if $\frac{1}{\alpha_{M_1}}+\cdots+\frac{1}{\alpha_{M_N}}<1$.
\item[\rm (iv)] $\mathcal{L}(l_{M_1}\times \cdots\times l_{M_N}, l_{M_{N+1}})$ is SSD if and only if $\frac{1}{\alpha_{M_1}}+\cdots+\frac{1}{\alpha_{M_N}}<\frac{1}{\beta_{M_{N+1}}}$.
\item[\rm (v)] $\mathcal{L}(d(w_1,p_1)\times \cdots\times d(w_N, p_N))$ is SSD if and only if $\frac{1}{p_1}+\cdots+\frac{1}{p_N}<1$.
\item[\rm (vi)] $\mathcal{L}(d(w_1,p_1)\times \cdots\times d(w_N, p_N), l_{M_{N+1}})$ is SSD if and only if $\frac{1}{p_1}+\cdots+\frac{1}{p_N}<\frac{1}{\beta_{M_{N+1}}}$.
\end{enumerate}
\end{maincor}

As we did below the statement of Theorem~\ref{theoremA}, we make now some observations regarding the proof of the equivalence in Theorem~\ref{multilineartheoremA}. The equivalence (a)$\Leftrightarrow$(b) and the implication (b)$\Rightarrow$(c) are, as in the polynomial case, immediate.
Note that, in contrast with Theorem~\ref{theoremA}, we require that $X_1, \dots, X_N$ have Schauder bases, which is stronger than the compact approximation property hypothesis. The reason is that, the equivalence (c)$\Leftrightarrow$(d) is proved in \linebreak \cite[Theorem~1 and Corollary~2]{DimZal} under this stronger assumption. Hence, we only need to prove implications (d)$\Rightarrow$(e)$\Rightarrow$(a), which we leave for the next section.


\vspace{0.2cm}

We focus now on the (uniform) strong subdifferentiability of tensor products. As we related, in Theorems~\ref{theoremA} and \ref{multilineartheoremA}, the strong subdifferentiability of spaces of polynomials and multilinear mappings with property $\Loo$, we will derive some differentiability properties of symmetric (respectively, full) projective tensor products from the \emph{polynomial} (respectively, \emph{multilinear}) $\Lpp$. We consider the following subsets of $X_1 \pten \cdots \pten X_N$ and $\sten X$, respectively, 
\begin{equation*}
U:= \Big\{ x_1 \otimes \cdots \otimes x_N: \|x_1\|=\cdots=\|x_N\|=1 \Big\}\subseteq S_{X_1 \pten \cdots \pten X_N}
\end{equation*}
and
\begin{equation*}
U_s := \Big\{ \otimes^N x: \|x\| = 1 \Big\} \subseteq S_{\sten X}, 
\end{equation*}
and we invoke Definitions~\ref{definition:SSD} and \ref{definition:uniform-SSD} on the (uniformly) SSD at a subset of the unit sphere. 

\begin{mainth} \label{theoremC} In the symmetric projective tensor setting, the following results hold true.
\begin{enumerate}
\itemsep0.3em 
	 \item[(i)] $\sten \ell_2$ is USSD on $U_s$ for $N \in \mathbb{N}$. 
	 \item[(ii)] ${\ensuremath{c_0 \widehat{\otimes}_{\pi_s}}} c_0$ is SSD on $U_s$ (in the complex case).
\end{enumerate}
In the (full, not symmetric) projective tensor setting, we have the following.
\begin{enumerate}
\itemsep0.3em 
\item[(iii)] $\ell_2\pten \stackrel{N}{\cdots} \pten \ell_2$ is USSD on $U$ for $N \in \mathbb{N}$.
\item[(iv)] $c_0\pten c_0$ is SSD on $U$ (in the complex case).
\item[(v)] $\ell_1^N \pten Y$ is SSD if and only if $Y$ is SSD. 
\end{enumerate}
\end{mainth}

\section{On the strong subdifferentiability of $\mathcal{P}(^N X, Y^*)$ and $\mathcal{L}(X_1\times\cdots\times X_N, Y^*)$} \label{Section:TheoremAandB}

In this section we prove Theorems~\ref{theoremA} and \ref{multilineartheoremA} and their respective corollaries. In Subsection~\ref{strongly exposition}, we make a deep analysis of the existing relation between the $N$-homogeneous polynomial ${\bf L}_{o,o}$ and strong subdifferentiability. We show that, under the assumption of $X$ being uniformly convex (we do not require the CAP nor the sequential Kadec-Klee property), the $N$-ho\-mo\-ge\-neous polynomial ${\bf L}_{o,o}$ is equivalent to the strongly exposition of the closed convex hull of elementary tensor products of norming points (a property which is formally stronger than strong subdifferentiability of the norm of $\mathcal{P}(^NX, Y^*)$). Finally, in Subsection~\ref{diagram} we make a diagram showing the implications between all the properties appearing in the previous sections, and the hypotheses needed in each implication.

\subsection{Proofs of Theorem \ref{theoremA} and Theorem \ref{multilineartheoremA}} We focus first in the proofs of Theorems~\ref{theoremA} and \ref{multilineartheoremA}. As we already mentioned in the previous section, we only need to prove implications (d)$\Rightarrow$(e)$\Rightarrow$(a) on both theorems. 

\begin{proof}[Proof of Theorem~\ref{theoremA}]
We begin with (d)$\Rightarrow$(e). We are going to prove that if $X$ is reflexive and has the sequential Kadec-Klee property, then the pair $(X, Y^*)$ has the \emph{$N$-homogeneous polynomial $\Loo$ for weakly sequentially continuous polynomials}, that is, given $\e > 0$ and \linebreak$P \in \mathcal{P}_{wsc}(^NX, Y^*)$ with $\|P\| = 1$, there exists $\eta(\e, P) > 0$ such that whenever $x \in S_X$ satisfies $\|P(x)\| > 1 - \eta(\e, P)$, there exists $x_0 \in S_X$ such that 
		$$
		\|P(x_0)\| = 1\quad \text{and} \quad \|x_0 - x\| < \e.
		$$
This, together with the hypothesis in (d), gives (e). We argue by contradiction. Suppose that there are $\e_0 > 0$, $P_0 \in \mathcal{P}_{wsc}(^N X, Y^*)$ with $\|P_0\| = 1$, and $(x_n)_{n \in \N}$ such that 
\begin{equation} \label{ineq1} 
1 \geq \|P_0(x_n)\| \geq 1 - \frac{1}{n} \ \ \ \mbox{and} \ \ \ \dist(x_n, \NA(P_0)) \geq \e_0 > 0.
\end{equation}
Since $X$ is reflexive, we may (and we do) assume that there exists $x_0 \in B_X$ such that $x_n \stackrel{w}{\longrightarrow} x_0$. Given that $P_0$ is weakly sequentially continuous, we have $P_0 (x_n) \stackrel{\|\cdot\|}{\longrightarrow} P_0(x_0)$. By using (\ref{ineq1}), we get that $\|P_0(x_0)\| = 1$ and, therefore, $x_0 \in S_X$. Now, $x_n \stackrel{w}{\longrightarrow} x_0$ and $\|x_n\|\rightarrow \|x_0\|$ and then, as $X$ has the sequential Kadec-Klee property, $(x_n)_{n=1}^{\infty}$ converges to $x_0$ in norm, which is a contradiction because $x_0 \in \NA(P)$.

Now we focus on the implication (e)$\Rightarrow$(a). Given a norm-one polynomial $P\in \mathcal{P}(^NX, Y^*)$ and $\varepsilon>0$, we want to find $\delta >0$ such that
\begin{equation*}
\frac{ \Vert P + tQ\Vert -1}{t}- \tau(P,Q) < \varepsilon 
\end{equation*} 
for every $0<t<\delta$ and every $Q\in \mathcal{P}(^N X, Y^*)$ with $\Vert Q\Vert =1$ (see \eqref{eq:SSD} for the definition of $\tau(P,Q)$). Since $Y$ is uniformly convex, in view of the characterization of uniform convexity given in \cite[Theorem~2.1]{KL}, we can take $0<\tilde{\eta}(\varepsilon)<\varepsilon$ such that, if $(y^*, y_0)\in S_{Y^*}\times S_{Y}$ satisfy $|y^*(y_0)|>1-\tilde{\eta}(\varepsilon)$, then there exist $y_1\in S_Y$ such that $|y^*(y_1)|=1$ and $\|y_1-y_0\|<\varepsilon$. 
We will see that $$\delta=\frac{\eta\left(2^{-1}\tilde{\eta}\left(\frac{\e}{2}\right), P\right)}{2}$$ works for our purposes, where $\eta(\varepsilon, P)>0$ is the one in the hypothesis (e). Observe that, without loss of generality, we may assume that $0<\eta(\e,P)<\e$. In particular, $\delta < 2^{-2}\tilde{\eta}\left(\frac{\e}{2}\right)$. For any $Q$ and $0<t<\delta$ fixed, take $x_t\in S_X$ such that $\|(P+tQ)(x_t)\|=\Vert P+tQ\Vert$. Such an $x_t$ exists because the hypothesis in (e) implies that every polynomial attains its norm. Then, we have
\begin{eqnarray*}
\|P(x_t)\| &=& \|(P+tQ-tQ)(x_t)\| = \|(P+tQ)(x_t) - tQ(x_t)\|\\
&\geq& \Vert P + tQ\Vert -t \geq 1- t- t > 1-2\delta > 1- \eta\left(2^{-1}\tilde{\eta}\left(\frac{\e}{2}\right), P\right),
\end{eqnarray*}
and, by hypothesis, there is $z\in S_X$ such that 
$$
\|P(z)\|=1 \quad \text{and}\quad \Vert x_t- z\Vert < 2^{-1}\tilde{\eta}\left(\frac{\e}{2}\right).
$$
Since $(P+tQ)(x_t)\in Y^*$ is norm attaining, we can consider $y_t\in S_Y$ such that 
$$
(P+tQ)(x_t)(y_t)=\|(P+tQ)(x_t)\|=\|P+tQ\|.
$$
Then,
\begin{eqnarray}
\frac{\Vert P + tQ\Vert -1}{t}- \tau(P,Q) 
&=& \frac{\operatorname{Re}[(P+tQ)(x_t)(y_t)]- 1}{t}-\tau(P,Q)\nonumber \\
&\leq& \frac{\operatorname{Re}[(P+tQ)(x_t)(y_t)]- \operatorname{Re}[P(x_t)(y_t)]}{t}-\tau(P,Q)\nonumber \\
&=& \operatorname{Re}[Q(x_t)(y_t)]-\tau(P,Q).\label{tau}
\end{eqnarray}
Now, from the inequalities
\begin{eqnarray*}
\operatorname{Re}[P(x_t)(y_t)] &=& \operatorname{Re} [(P+tQ-tQ)(x_t)(y_t)] = \operatorname{Re} [(P+tQ)(x_t)(y_t)] - \operatorname{Re} [tQ(x_t)(y_t)]\\
&\geq& \Vert P + tQ\Vert -t \geq 1- 2t> 1-2\delta> 1-2^{-1}\tilde{\eta}\left(\frac{\varepsilon}{2}\right)
\end{eqnarray*}
and
\begin{eqnarray*}
\left| \operatorname{Re}[P(z)(y_t)] - \operatorname{Re}[P(x_t)(y_t)]\right|\leq \|P(z)-P(x_t)\|\leq \|z-x_t\|\leq 2^{-1}\tilde{\eta}\left(\frac{\e}{2}\right),
\end{eqnarray*}
we deduce that
$$
|P(z)(y_t)|\geq \operatorname{Re}[P(z)(y_t)]> \operatorname{Re}[P(x_t)(y_t)] - 2^{-1}\tilde{\eta}\left(\frac{\e}{2}\right) > 1-\tilde{\eta}\left(\frac{\e}{2}\right).
$$
Then, there exist $y\in S_Y$ such that $P(z)(y)=\|P(z)\|=1$ and $\|y-y_t\|<\e/2$. Finally, since
$$
\tau(P,Q)\geq \operatorname{Re}[Q(z)(y)],
$$
going back to \eqref{tau} we see that
\begin{eqnarray*}
\frac{\Vert P + tQ\Vert -1}{t}- \tau(P,Q) &\leq& \operatorname{Re}[Q(x_t)(y_t)] - \operatorname{Re}[Q(z)(y)]\\
&\leq& \operatorname{Re}[Q(x_t)(y_t)] - \operatorname{Re}[Q(z)(y_t)] + \operatorname{Re}[Q(z)(y_t)] - \operatorname{Re}[Q(z)(y)]\\
&\leq& \|x_t-z\| + \|y_t-y\| < \frac{\e}{2}+\frac{\e}{2}=\e
\end{eqnarray*}
whenever $0<t<\delta$, which is the desired statement.
\end{proof}

\begin{proof}[Proof of Theorem~\ref{multilineartheoremA}]
Let us begin with the implication (d)$\Rightarrow$(e). Arguing by contradiction, exactly as in the proof of Theorem~\ref{theoremA}, we get that the pair $(X_1\times \cdots\times X_{N-1}, X_N^*)$ has the \emph{$\Loo$ for weakly sequentially continuous mappings}, that is, the statement in Definition~\ref{def Lpp and Loo} (ii) holds for every $A\in \mathcal{L}_{wsc}(X_1\times \cdots\times X_{N-1}, X_N^*)$. Since, by hypothesis we know that \linebreak$\mathcal{L}(X_1\times \cdots\times X_{N-1}, X_N^*)=\mathcal{L}_{wsc}(X_1\times \cdots\times X_{N-1}, X_N^*)$, we deduce that the pair $(X_1\times \cdots\times X_{N-1}, X_N^*)$ has the $\Loo$. Now, let us show that this implies that the pair $(X_1\times \cdots\times X_{N}, \K)$ has the $\Loo$. Take $\e>0$ and a norm-one $N$-linear form $A\in \mathcal{L}(X_1\times \cdots\times X_{N})$, and let $\eta(\cdot,\tilde{A})>0$ be the one in the definition of property $\Loo$ for the pair $(X_1\times \cdots\times X_{N-1}, X_N^*)$, where \linebreak $\tilde{A}(x_1, \dots, x_{N-1})(x_N)=A(x_1, \dots, x_{N})$. As $X_N$ is uniformly convex, there exist $0<\tilde{\eta}(\e)<\e$ such that if $(x_N^*, x_N) \in S_{X_N^*}\times S_{X_N}$ satisfy $|x_N^*(x_N)|>1-\tilde{\eta}(\e)$, then there exist $x_N^1\in S_{X_N}$ such that $|x_N^*(x_N^1)|=1$ and $\|x_N^1-x_N\|<\e$. Suppose that
$$
|A(x_1,\dots, x_N)|>1-\frac{1}{2}\eta\left(\frac{\tilde{\eta}(\e)}{2N}, \tilde{A}\right).
$$
On the one hand, we have 
$$
\|\tilde{A}(x_1,\dots, x_{N-1})\|>1-\eta\left(\frac{\tilde{\eta}(\e)}{2N}, \tilde{A}\right)
$$
and, by the $\Loo$ property for the pair $(X_1\times \cdots\times X_{N-1}, X_N^*)$, there exist \linebreak $(x_1^1,\dots, x_{N-1}^1)\in S_{X_1}\times\cdots\times S_{X_{N-1}}$ such that
$$
\|\tilde{A}(x_1^1,\dots, x_{N-1}^1)\|=1 \quad \text{and}\quad \|x_i-x_i^1\|< \frac{\tilde{\eta}(\e)}{2N}, \quad i=1,\dots, N-1.
$$
As a consequence, 
\begin{eqnarray*}
\left| |\tilde{A}(x_1^1,\dots, x_{N-1}^1)(x_N)| - |\tilde{A}(x_1,\dots, x_{N-1})(x_N)| \right| &\leq& |A(x_1^1,\dots, x_{N-1}^1, x_N) - A(x_1,\dots, x_{N-1}, x_N)|\\
&\leq& \|x_{N-1}^1-x_{N-1}\|+\cdots+\|x_1^1-x_1\|\\
&\leq& \frac{\tilde{\eta}(\e)}{2}
\end{eqnarray*}
and, hence,
$$
|\tilde{A}(x_1^1,\dots, x_{N-1}^1)(x_N)|>|\tilde{A}(x_1,\dots, x_{N-1})(x_N)| - \frac{\tilde{\eta}(\e)}{2} > 1 - \tilde{\eta}(\e).
$$
Then, there exist $x_N^1\in S_{X_N}$ such that
$$
|\tilde{A}(x_1^1,\dots, x_{N-1}^1)(x_N^1)|=1 \quad \text{and} \quad \|x_N-x_N^1\|<\e.
$$
In sum,
$$
|A(x_1^1,\dots, x_{N-1}^1, x_N^1)|=1 \quad \text{and} \quad \|x_i-x_i^1\|<\e \quad i=1,\dots, N,
$$
as desired.

Now we prove (e)$\Rightarrow$(a). Given $\e > 0$ and $A \in \mathcal{L}(X_1 \times \cdots \times X_N)$, we will find $\delta >0$ such that
\begin{equation*}
\frac{ \Vert A + tL\Vert -1}{t}- \tau(A,L) < N\varepsilon 
\end{equation*} 
for every $0<t<\delta$ and every $L\in \mathcal{L}(X_1 \times \cdots \times X_N)$ with $\Vert L\Vert =1$ (recall (\ref{eq:SSD}) and (\ref{eq:SSD1})). We will see that $\delta :=\frac{\eta(\e, A)}{2} > 0$ does the job, where $\eta(\varepsilon, A)>0$ is the one in the definition of property $\Loo$ (see Definition~\ref{def Lpp and Loo} (ii)). For any $L$ and $0<t<\delta$ fixed, take $\mathbf{x}_t=(x_1^t,\ldots, x_N^t) \in S_{X_1}\times \cdots \times S_{X_N}$ such that $(A+tL)(\mathbf{x}_t)=\Vert A+tL\Vert$. Then, we have
\begin{eqnarray*}
\operatorname{Re} A(\mathbf{x}_t) &=& \operatorname{Re} [(A+tL-tL)(\mathbf{x}_t)] = \operatorname{Re} [(A+tL)(\mathbf{x}_t)] - \operatorname{Re} [tL(\mathbf{x}_t)]\\
&\geq& \Vert A + tL\Vert -t \geq 1- t- t > 1-\eta(\varepsilon, A),
\end{eqnarray*}
and, by hypothesis, there exists $\mathbf{z}_t=(z_1^t,\ldots, z_N^t) \in S_{X_1}\times \cdots \times S_{X_N}$ with $\Vert x_i^t- z_i^t\Vert < \varepsilon$ and $A(\mathbf{z}_t)=1$ That is, the linear functional defined as the evaluation in $\mathbf{z}_t$ belongs to the set $D(A)$ of support functionals at $A$. Then,
\begin{eqnarray*}
\frac{\Vert A + tL\Vert -1}{t}- \tau(A,L) 
&=& \frac{\operatorname{Re}[(A+tL)(\mathbf{x}_t)]- 1}{t}-\tau(A,L)\nonumber \\
&\leq& \frac{\operatorname{Re}[(A+tL)(\mathbf{x}_t)]- \operatorname{Re}[A(\mathbf{x}_t)]}{t}-\tau(A,L)\nonumber \\
&=& \operatorname{Re}[L(\mathbf{x}_t)]-\tau(A,L)\nonumber \\
&\leq& \operatorname{Re}[L(\mathbf{x}_t)] - \operatorname{Re} [L(\mathbf{z}_t)] \\
&<& N\varepsilon, \
\end{eqnarray*}
which proves the desired statement.
\end{proof}

Now, we move towards the proof of Corollary~\ref{corollaryA}. In view of equivalence (a)$\Leftrightarrow$(d) in Theorem~\ref{theoremA}, and taking into account that all the spaces considered in Corollary~\ref{corollaryA} satisfy the hypotheses of the theorem (see Subsection~\ref{Orlicz Lorentz}, where we briefly listed some known properties of Orlicz and Lorentz sequence spaces), we only need to check that the space of \linebreak $N$-homogeneous polynomials coincide with the space of weakly sequentially continuous \linebreak $N$-homogeneous polynomials. The following remark will be useful in the proof of the corollary.

\begin{remark}\label{remark not wsc}
In \cite[Remark~3]{JarPriZal} the authors show that if $X$ has a quotient isomorphic to $\ell_p$ and $N\geq p$, then $\mathcal{P}(^N X)\neq \mathcal{P}_{wsc}(^N X)$. Following the same ideas we can see that, if $\ell_p$ is isomorphic to a quotient of $X$, $\ell_q$ is isomorphic to a subspace of $Y$ and $Nq\geq p$, then $\mathcal{P}(^N X, Y)\neq \mathcal{P}_{wsc}(^N X, Y)$. Indeed, let $\pi\colon X\to \ell_p$ be a quotient map and take a bounded sequence $(x_n)_n$ in $X$ such that $\pi(x_n)=e_n$, where $\{e_n\}$ is the canonical basis of $\ell_p$. Since $X$ does not contain a copy of $\ell_1$, by Rosenthal's theorem we know that $(x_n)_n$ admits a weakly Cauchy subsequence $(x_{n_j})_j$. Consider the weakly null sequence in $X$ given by $y_j=x_{n_{2j}}-x_{n_{2j+1}}$ and the polynomial $Q\in \mathcal{P}(^N\ell_p, \ell_q)$ defined by
$$
Q(a_1, a_2, \dots, a_j, \dots)=(a_{n_2}^N, a_{n_4}^N, \dots, a_{n_{2j}}^N, \dots)
$$
(here we use the fact that $Nq\geq p$). Finally, let $i: \ell_q \to Y$ be an isomorphism onto its image and consider $P=i\circ Q \circ \pi \in \mathcal{P}(^N X, Y)$. Noting that $\|P(y_j)\|=\|i(Q(e_{n_{2j}}-e_{n_{2j+1}}))\|\not\to 0$, we conclude that $P$ is not weakly sequentially continuous.
\end{remark}
 
\begin{proof}[Proof of Corollary~\ref{corollaryA}]
For $\ell_p$-spaces it is known that $\mathcal{P}(^N\ell_p)=\mathcal{P}_{wsc}(^N\ell_p)$ if and only if $N<p$ and that $\mathcal{P}(^N\ell_p,\ell_q)=\mathcal{P}_{wsc}(^N\ell_p,\ell_q)$ if and only if $Nq<p$ (see for example \cite[Chapter~2.4]{Din}). This gives items (i) and (ii). 

Let us prove items (iii) and (iv). On the one hand, by 
\cite[Theorem~2.5 and Corollary~2.6]{GJ} we have that $\mathcal{P}(^N l_{M_1})=\mathcal{P}_{wsc}(^N l_{M_1})$ if $N<l(l_{M_1})=\alpha_{M_1}$ and that $\mathcal{P}(^N l_{M_1}, l_{M_2})=\mathcal{P}_{wsc}(^N l_{M_1}, l_{M_2})$ if $N \beta_{M_2}=N u(l_{M_2})< l(l_{M_1})=\alpha_{M_1}$. This gives the "if" implication in items (iii) and (iv). On the other hand, suppose that $N\geq \alpha_{M_1}$. Putting $p=\alpha_{M_1}$ we have that $\ell_p$ is isomorphic to a quotient space of $l_{M_1}$ (see \cite[Theorem~1 and Corollary~1]{LTIII}). Hence, from Remark~\ref{remark not wsc}, we deduce that $\mathcal{P}(^N l_{M_1})\neq \mathcal{P}_{wsc}(^N l_{M_1})$ and, consequently, $\mathcal{P}(^N l_{M_1})$ is not SSD. In the vector-valued case, suppose that $N \beta_{M_2}\geq \alpha_{M_1}$ and put $p=\alpha_{M_1}$ and $q=\beta_{M_2}$. Again by \cite[Theorem~1 and Corollary~1]{LTIII} we have that $\ell_p$ is isomorphic to a quotient space of $l_{M_1}$ and $\ell_q$ is isomorphic to a subspace of $l_{M_2}$. Then, in virtue of Remark~\ref{remark not wsc} we have $\mathcal{P}(^N l_{M_1}, l_{M_2})\neq \mathcal{P}_{wsc}(^N l_{M_1}, l_{M_2})$, which is the desired statement.
  
Finally, we sketch the proof of item (v) and (vi). If $N<l(d(w,p))=p$ then \linebreak $\mathcal{P}(^N d(w,p))=\mathcal{P}_{wsc}(^N d(w,p))$ in virtue of the cited results in \cite{GJ}. When $N\geq p$, given that $d(w,p)$ has a quotient isomorphic to $\ell_p$ (see \cite[Proposition~4]{LTII}), by Remark~\ref{remark not wsc} we have that $\mathcal{P}(^N d(w,p))\neq\mathcal{P}_{wsc}(^N d(w,p))$. The proof of (vi) is analogous to that of (iv).

Alternatively, for the only if parts, one could use that under the hypothesis of Remark~\ref{remark not wsc}, $\mathcal{P}(^N \ell_p, \ell_q)$ is a subspace of $\mathcal{P}(^N X, Y)$. Thus, if $Nq\geq p$, this space can not be reflexive.
\end{proof}

\begin{proof}[Proof of Corollary~\ref{corollaryB}]
As in the proof of Corollary~\ref{corollaryA} we only need to check that, in each case, the space of $N$-linear mappings coincide with the space of weakly sequentially continuous $N$-linear mappings (or, equivalently, that the space of $N$-linear mappings is reflexive). Note that, as in Corollary~\ref{corollaryA}, the spaces considered satisfy the hypotheses of Theorem~\ref{multilineartheoremA}. Items (i) and (ii) follow from the fact that $\mathcal{L}(\ell_{p_1}\times \cdots\times \ell_{p_N}, \ell_q)=\mathcal{L}_{wsc}(\ell_{p_1}\times \cdots\times \ell_{p_N}, \ell_q)$ if and only if $\frac{1}{p_1}+\cdots+\frac{1}{p_N}<\frac{1}{q}$ (see \cite[Chapter~2.4]{Din}). The "if" part of items (iii) and (iv) follow from \linebreak \cite[Lemma]{DimZal}, where it is proved that if
$$
\frac{1}{l(X_1)}+\cdots+\frac{1}{l(X_N)}<\frac{1}{u(X_{N+1})},
$$
then every $N$-linear mapping in $\mathcal{L}(X_1\times \cdots\times X_N, X_{N+1})$ is weakly sequentially continuous. 

The "only if" implication follows applying a \emph{multilinear} version of Remark~\ref{remark not wsc}. Specifically, it can be proved that if $\ell_{p_i}$, $i=1, \dots, N$, is isomorphic to a quotient of $X_i$, $\ell_q$ is isomorphic to a subspace of $X_{N+1}$ and 
$$
\frac{1}{p_1}+\cdots+\frac{1}{p_N}\geq\frac{1}{q},
$$
then $\mathcal{L}(X_1\times \cdots\times X_N, X_{N+1})\neq \mathcal{L}_{wsc}(X_1\times \cdots\times X_N, X_{N+1})$. Or, as before, one could see that $\mathcal{L}(\ell_{p_1}\times \cdots\times \ell_{p_N}, \ell_{q})$ is a subspace of $\mathcal{L}(X_1\times \cdots\times X_N, X_{N+1})$. The details are left to the reader.
Finally, items (v) and (vi) follow applying the same arguments.
\end{proof}

\subsection{Strongly exposition of polynomials}\label{strongly exposition}
Before carrying out a deeper analysis of the existing relation between the $\NLoo$ and strong subdifferentiability, let us set some definitions needed for this subsection. For an $N$-homogeneous polynomial $P \in \mathcal{P}(^N X, Y^*)$ with $\|P\| = 1$, we define the set
\begin{equation} \label{set-C(P)}
C(P):= \overline{\co} \Big\{ (\otimes^N x) \otimes y : x \in S_X, y \in S_Y, \text{ and } P(x)(y) = 1 \Big\},
\end{equation}
which clearly satisfies $C(P)\subseteq D(P)=\{\varphi\in S_{\mathcal{P}(^N X, Y^*)^*}: \varphi(P)=1\}$. Note that when $N=1$, $X$ is reflexive and $Y= \mathbb{K}$ we have that $C(P) = D(P)$, but the equality does not hold in general. 
As we already mentioned in Subsection~\ref{section SSD}, the norm of $\mathcal{P}(^N X, Y^*)$ is SSD at $P$ if and only if $P$ strongly exposes $D(P)$.
Our aim in this section, is to establish some relations between the $N$-homogeneous polynomial $\Loo$, strong exposition of the set $C(P)$ and strong subdifferentiability of the norm of $\mathcal{P}(^N X, Y^*)$. As a byproduct, we obtain a result on denseness of norm attaining symmetric tensor products, in the same line of \cite{DGJR}.
We begin noting that, since $C(P)\subseteq D(P)$, if $P$ strongly exposes $C(P)$, then it strongly exposes $D(P)$.

\begin{remark}\label{remark strongly C(P)} Let $X, Y$ be Banach spaces and $P\in S_{\mathcal{P}(^N X, Y^*)}$. If $P$ strongly exposes $C(P)$, then the norm of $\mathcal{P}(^N X, Y^*)$ is strongly subdifferentiable at $P$. In particular, if $P$ strongly exposes $C(P)$ for every $P\in S_{\mathcal{P}(^N X, Y^*)}$, then the norm of $\mathcal{P}(^N X, Y^*)$ is strongly subdifferentiable.
\end{remark}

\begin{proof}By hypothesis, if $(\phi_n) \subseteq \mathcal{P}(^N X, Y^*)^*$ with $\| \phi_n \| \leq1$ satisfies that $\re \phi_n (P) \rightarrow 1$, then we have that $\dist (\phi_n, C(P)) \rightarrow 0$. Given that $C(P) \subseteq D(P)$, we have that $\dist( \cdot, D(P)) \leq \dist(\cdot, C(P))$ and, therefore, $\dist(\phi_n, D(P)) \rightarrow 0$. Hence, the norm of $\mathcal{P}(^N X, Y^*)$ is SSD.
\end{proof}

Implication (e)$\Rightarrow$(a) of Theorem~\ref{theoremA} shows that, whithout any assumption on the space $X$, the $N$-homogeneous polynomial $\Loo$ of the pair $(X,\mathbb{K})$ imply that $\mathcal{P}(^NX)$ is SSD (the same holds for the pair $(X, Y^*)$ whenever $Y$ is uniformly convex). In other words, the $N$-homogeneous polynomial $\Loo$ is stronger than strong subdifferentiability. 
In view of the previous remark, it is natural to ask if there is a relation between the $N$-homogeneous polynomial $\Loo$ and strongly exposition of the set $C(P)$ for every $P\in S_{\mathcal{P}(^N X, Y^*)}$. In Theorem~\ref{prop-star-equivalent-Loo} below, we prove that if the underlying spaces are uniformly convex, these properties are equivalent. 
On the way there we prove an auxiliary lemma, which relates the $N$-homogeneous polynomial $\Loo$ 
with the denseness of norm attaining symmetric tensors. This result, interesting by its own, should be compared with \cite[Theorem~3.8]{DGJR}. Recall that $z \in \sten X$ attains its projective symmetric norm if there are bounded sequences $(\lambda_n)_{n=1}^{\infty} \subseteq \K$ and $(x_n)_{n=1}^{\infty} \subseteq B_X$ such that $\|z\|_{\pi_s, N} = \sum_{n=1}^{\infty} |\lambda_n|$ and $z = \sum_{n=1}^{\infty} \lambda_n \otimes^N x_n$. In such a case, we say that the tensor $z$ is norm-attaining. We denote by $\NA_{\pi} (\sten X)$ the set of all $z \in \sten X$ such that $z$ attains its projective symmetric norm. Analogously, $z \in (\sten X) \pten Y$ attains its projective norm if there are bounded sequences $(\lambda_n)_{n=1}^{\infty} \subseteq \K$, $(x_n)_{n=1}^{\infty} \subseteq B_X$ and $(y_n)_{n=1}^{\infty} \subseteq B_Y$ such that $z = \sum_{n=1}^{\infty} \lambda_n (\otimes^N x_n)\otimes y_n$ with $\|z\|_{\pi} = \sum_{n=1}^{\infty} |\lambda_n|$. As expected, we denote $\NA_{\pi} ((\sten X)\pten Y)$ the set of norm attaining tensors.

\begin{lemma} \label{NA-symm1} Let $X, Y$ be reflexive Banach spaces.
\begin{enumerate}
\itemsep0.3em 
  \item[(i)] If the pair $(X,\mathbb{K})$ has the $N$-homogeneous polynomial $\Loo$, then 
\begin{equation*} 	
	\overline{\NA_{\pi}(\sten X)}^{\|\cdot\|_{\pi, s, N}} = \sten X.
\end{equation*} 
Moreover, given $\e >0$ and $z\in S_{\sten X}$, if $P_0\in S_{\mathcal{P}(^NX)}$ satisfies $1 = \langle P_0, z \rangle$, there exist $w\in \NA_{\pi}(\sten X)$ such that $\|w-z\|_{\pi_s, N}<\e$ and $\| w \|_{\pi_s, N} = \langle P_0, w \rangle$.
  \item[(ii)] If $Y$ is uniformly convex and the pair $(X,Y^*)$ has the $N$-homogeneous polynomial $\Loo$, then 
\begin{equation*} 	
	\overline{\NA_{\pi}((\sten X)\pten Y)}^{\|\cdot\|_{\pi}} = (\sten X)\pten Y.
\end{equation*} 
Moreover, given $\e >0$ and $z\in S_{(\sten X)\pten Y}$, if $P_0\in S_{\mathcal{P}(^NX, Y^*)}$ satisfies $1 = \langle P_0, z \rangle$, there exist $w\in \NA_{\pi}((\sten X)\pten Y)$ such that $\|w-z\|_\pi<\e$ and $\| w \|_\pi = \langle P_0, w \rangle$.
\end{enumerate}
\end{lemma}

\begin{proof} 
For the proof of (i), we follow ideas from \cite[Proposition 4.3]{DJRR}. Let $z \in \sten X$ with $\| z \| =1$ and $\e >0$ be given and fix $\delta>0$ (which will be chosen appropriately later). We can find $P_0 \in \mathcal{P} (^N X)$ such that $\| P_0 \| = \langle P_0, z \rangle = 1$. Let $z' = \sum_{j=1}^m \lambda_j \otimes^N x_j$ with $\lambda_j \geq 0$, $(x_j)_j\subset B_X$ and $m \in \N$ be such that 
$$
\sum_{j=1}^m \lambda_j \leq 1 + \eta(\de, P_0)^2 \quad \text{and}\quad \| z- z' \| < \eta(\de, P_0)^2,
$$ 
where $\eta(\de, P_0)>0$ is the one given in the definition of the $N$-homogeneous polynomial $\Loo$ of the pair $(X,\mathbb{K})$. 
Note that 
\[
1+\eta(\de, P_0)^2 \geq \sum_{j=1}^m \lambda_j \geq \re \sum_{j=1}^m \lambda_j \langle P_0, \otimes^N x_j \rangle > 1-\eta(\de, P_0)^2,
\]
which implies that $\sum_{j=1}^m \lambda_j (1-\re \langle P_0, \otimes^N x_j\rangle ) < 2 \eta(\de, P_0)^2$. 
Now, defining 
\begin{equation*} 
	A = \Big\{ i \in \{1,\ldots, m\} : 1 - \re \langle P_0, \otimes^N x_i \rangle < \eta(\de, P_0) \Big\}
\end{equation*} 	
and noting that 
	\[
	\eta(\de, P_0) \sum_{j \in \{1,\ldots,m\} \setminus A} \lambda_j \leq \sum_{j \in \{1,\ldots,m\} \setminus A} \lambda_j (1-\re\langle P_0, \otimes^N x_j\rangle ) < 2 \eta(\de, P_0)^2,
	\]
we deduce $\sum_{j \in \{1,\ldots,m\} \setminus A} \lambda_j < 2 \eta(\de,P_0)$. By the $N$-homogeneous polynomial $\Loo$ of the pair $(X;\mathbb{K})$, for each $j \in A$ we can take $u_j \in S_X$ so that 
$$
|P_0 (u_j)| = 1\quad \text{and}\quad \|u_j - x_j \| < \de.
$$ 
Write $P_0 (u_j) = \theta_j \in \mathbb{T}$ for each $j \in A$. Note that $\| \otimes^N u_j - \otimes^N x_j \| < \frac{N^{N+1}}{N!} \de$ for each $j \in A$ (see, for instance, \cite[Lemma 2.2]{DGJR}). Moreover, for each $j \in A$, 
	\[
	\re \theta_j = \re \langle P_0, \otimes^N u_j \rangle > (1-\eta(\de, P_0)) - \frac{N^{N+1}}{N!} \de 
	\] 
	which implies that $|1 - \theta_j | < \sqrt{2 (\eta(\de, P_0) + \frac{N^{N+1}}{N!} \de) }$ for each $j \in A$ (if we consider the real scalar field, then $\theta_j$ would be $1$). If we let $w = \sum_{j\in A} \lambda_j \theta_j^{-1} \otimes^N u_j$, then 
	\begin{eqnarray*}
	\| w - z\| &\leq& \|w-z'\|+\|z'-z\|\\
	&\leq& \left\| \sum_{j\in A} \lambda_j \theta_j^{-1} \otimes^N u_j- \sum_{j=1}^m \lambda_j \otimes^N x_j	\right\| + \eta(\de, P_0)^2 \\
	&\leq& \left\| \sum_{j\in A} \lambda_j \theta_j^{-1} \otimes^N u_j- \sum_{j\in A} \lambda_j \otimes^N u_j \right\| + \left\| \sum_{j\in A} \lambda_j (\otimes^N u_j -\otimes^N x_j)  \right\| + 2 \eta(\de, P_0) + \eta(\de, P_0)^2\\
	&<& (1+\eta(\de, P_0)^2 ) \left(\sqrt{2 \left(\eta(\de, P_0) + \frac{N^{N+1}}{N!} \de\right) } + \frac{N^{N+1}}{N!} \de \right) + 2 \eta(\de, P_0) + \eta(\de, P_0)^2.
	\end{eqnarray*} 
On the one hand, choosing $\de>0$ small enough we obtain $\|w-z\|<\e$.	On the other hand, noticing that $\| w \| = \langle P_0, w \rangle = \sum_{j \in A} \lambda_j$ we deduce that $w$ attains its norm. 

We briefly sketch the proof of (ii), which is analogous to the previous one. In what follows, we denote $\|\cdot\|$ both projective and symmetric projective norms, since it is clear by context. Let $\e >0$ and $z\in S_{(\sten X)\pten Y}$ be given, and consider $P_0\in S_{\mathcal{P}(^NX, Y^*)}$ such that $1 = \langle P_0, z \rangle$. Take $z'=\sum_{j=1}^m \left(\sum_{i=1}^n \lambda_{j,i} \otimes^N x_{j,i}\right)\otimes y_j$ with $\lambda_{j,i}>0$, $x_{j,i} \subset B_X$ and $y_j\in B_Y$ such that
$$
\sum_{j=1}^m \sum_{i=1}^n \lambda_{j,i} \leq 1 + \min\left\{\eta\left(\frac{\delta_Y(\de/2)}{2},P_0\right), \frac{\delta_Y(\de/2)}{2}\right\}^2=: 1 + \tilde{\eta}(\de, P_0)^2
$$
and $\|z-z'\|<\tilde{\eta}(\de, P_0)^2$,
where $\de>0$ is fixed (and chosen appropriately below) and $\delta_Y(\cdot)$ is the modulus of uniform convexity of $Y$. It can be seen that 
$$
\sum_{j=1}^m \sum_{i=1}^n \lambda_{j,i} (1-\re P_0(x_{j,i})(y_j)) < 2 \tilde{\eta}(\de, P_0)^2.
$$ 
Then, defining 
$$
	A = \Big\{ (j,i) \in \{1,\ldots, m\}\times \{1,\ldots, n\} : 1 - \re P_0(x_{j,i})(y_j) < \tilde{\eta}(\de, P_0) \Big\}
$$
we have 
$$
\sum_{(j,i)\notin A} \lambda_{j,i}<2\tilde{\eta}(\delta, P_0).
$$
Now, on the one hand, by the $N$-homogeneous polynomial $\Loo$ of the pair $(X,Y^*)$, for each $(j,i) \in A$ we can take $u_{j,i} \in S_X$ so that 
$$
\|P_0 (u_{j,i})\| = 1\quad \text{and}\quad \|u_{j,i} - x_{j,i} \| < \frac{\delta_Y(\de/2)}{2}.
$$ 
On the other hand, for each $(j,i) \in A$ we have
$$
|P_0(u_{j,i})(y_j)|>|P_0(x_{j,i})(y_j)|-\|u_{j,i}-x_{j,i}\|>1-\frac{\delta_Y(\de/2)}{2}-\frac{\delta_Y(\de/2)}{2}
$$
and, since $Y$ is uniformly convex, by \cite[Theorem~2.1]{KL} there exist $v_j\in S_Y$ such that
$$
|P_0(u_{j,i})(v_j)|=1\quad \text{and}\quad \|v_j-y_j\|<\de.
$$
Then, putting $\theta_{j,i}=P_0(u_{j,i})(v_j)$, the norm attaining tensor which approximates $z$ is $$w=\sum_{(j,i)\in A} \left(\sum_{i=1}^n \lambda_{j,i} \theta_{j,i}^{-1} \otimes^N u_{j,i}\right)\otimes v_j.$$
\end{proof} 
 
We are ready now to prove the mentioned equivalence between the $N$-homogeneous polynomial $\Loo$ and strong exposition of the set $C(P)$.

\begin{theorem} \label{prop-star-equivalent-Loo} Let $X$ and $Y$ be uniformly convex Banach spaces. The pair $(X, Y^*)$ has the $N$-homogeneous polynomial $\Loo$ if and only if $P$ strongly exposes $C(P)$ for every $P\in S_{\mathcal{P}(^N X, Y^*)}$.
\end{theorem}

\begin{proof}
Suppose first that $(X, Y^*)$ has the $N$-homogeneous polynomial $\Loo$. In view of implication (e)$\Rightarrow$(a) of Theorem~\ref{theoremA}, $\mathcal{P}(^N X, Y^*)$ is SSD (here we use the uniform convexity of $Y$). Hence, it is enough to show that $C(P)$ and $D(P)$ coincide for each $P \in \mathcal{P}(^N X, Y^*)$ with $\| P \| =1$. To this end, let $\phi \in D(P)$. Since $\mathcal{P} (^N X, Y^*)$ is reflexive we have $\phi \in (\sten X) \pten Y$ and, by Lemma~\ref{NA-symm1}, we know that given $\e >0$ there exists $\phi' = \sum_{j=1}^m \left(\sum_{i=1}^n \lambda_{j,i} \otimes^N u_{j,i}\right) \otimes v_j$ in $(\sten X) \pten Y$ satisfying
$$
\| \phi'\| = \langle P, \phi' \rangle = \sum_{j=1}^m \sum_{i=1}^n \lambda_{j,i}\quad \text{and}\quad \| \phi - \phi'\| < \e,
$$ 
where $\lambda_{j,i}>0$, $u_{j,i} \in S_X$ and $v_j \in S_Y$ for each $(j,i) \in \{1,\ldots, m\}\times \{1,\ldots, n\}$. This implies that $\phi^{''} := \sum_{j=1}^m \sum_{i=1}^n \left(\frac{\lambda_{j,i}}{\|\phi'\|} \otimes^N u_{j,i}\right) \otimes v_j \in C(P)$ and $\| \phi^{''} - \phi \| < 2\e$. Thus, $D(P) \subseteq C(P)$ and we are done. 

Let us see now the reverse implication. Suppose, by contradiction, that $(X, Y^*)$ does not satisfy the $N$-homogeneous polynomial $\Loo$. Then, there exist $P \in \mathcal{P}(^N X, Y^*)$ with $\|P\| = 1$ and $\e_0 > 0$ such that, for every $n \in \N$, there exists $x_n \in S_X$ with 
	\begin{equation} \label{ineq7}
	1 -\frac{1}{n} \leq \|P(x_n) \| \leq 1 \ \ \ \mbox{and} \ \ \ \dist(x_n, \NA(P)) \geq \e_0 > 0.
	\end{equation}
	Take $y_n \in S_Y$ and $\theta_n \in \mathbb{T}$ so that 
	\begin{equation*} 	
	|P(x_n)(y_n)| = P(x_n)( \theta_n y_n) > 1 -\frac{1}{n}. 
	\end{equation*} 
	Since $P$ strongly exposes $C(P)$, we have that $\dist( ( \otimes^N x_n) \otimes \theta_n y_n , C(P)) \rightarrow 0$. Without loss of generality, we assume that, for every $n \in \N$, 
	\begin{equation*}
	\dist\left(( \otimes^N x_n) \otimes \theta_n y_n , C(P)\right) \leq \frac{1}{n}.
	\end{equation*}
	For each $n \in \N$, let us take $\lambda_{1,n}, \ldots, \lambda_{s_n,n} > 0, u_{1,n}, \ldots, u_{s_n,n} \in S_X$ and $v_{1,n},\ldots, v_{s_n, n} \in S_Y$ to be such that 
	\begin{itemize}
	\itemsep0.3em
		\item[(I)] $\displaystyle \sum_{j=1}^{s_n} \lambda_{j,n} = 1$,
		\item[(II)] $P(u_{j,n}) (v_{j,n}) = 1$ for every $j=1,\ldots, s_n$ and 
		\item[(III)] $\displaystyle \left\| (\otimes^N x_n) \otimes \theta_n y_n - \sum_{j=1}^{s_n} \lambda_{j,n} (\otimes^N u_{j,n}) \otimes v_{j,n} \right\| < \frac{1}{n}$. 
	\end{itemize}
	Now, let us take $x_n^* \in S_{X^*}$ and $y_n^* \in S_{Y^*}$ to be such that $x_n^*(x_n) = 1$ and $y_n^* (y_n) = \theta_n^{-1}$ for every $n \in \N$. Then
	\begin{eqnarray*}
		\re \sum_{j=1}^{s_n} \lambda_{j,n} x_n^*(u_{j,n})^N y_n^* (v_{j,n}) &=& \re \left\langle (x_n^*)^N \otimes y_n^* , \sum_{j=1}^{s_n} \lambda_{j,n} (\otimes^N u_{j,n} )\otimes v_{j,n} \right\rangle \\
		&\stackrel{\text{(III)}}{\geq}& \re \left\langle (x_n^*)^N \otimes y_n^* , (\otimes^N x_n) \otimes \theta_n y_n \right\rangle - \frac{1}{n} \\
		&=& 1 - \frac{1}{n}.
	\end{eqnarray*}
	By a standard convex combination argument, there exists $t_n \in \{1,\ldots, s_n\}$ such that \linebreak $z_n := u_{t_n, n} \in S_X$ and $w_n := v_{t_n, n} \in S_Y$ satisfying 
	\begin{equation*}
	\re x_n^*(z_n)^N y_n^* (w_n) \geq 1 - \frac{1}{n}.
	\end{equation*}
	Taking $n_0 \in \N$ large enough so that 
	\begin{equation*}
	\left( 1 - \frac{1}{n_0} \right)^{\frac{1}{N}} > 1 - \delta_X(\e_0/2),
	\end{equation*}
	we have
	\begin{equation*}
	|x_{n_0}^* (z_{n_0})| > 1 - \delta_X (\e_0/2) 
	\end{equation*}
	which implies, by the uniform convexity of $X$, that $\| \theta z_{n_0} - x_{n_0}\| < \e_0$ for some $\theta \in \mathbb{T}$. But $\| P(\theta z_{n_0}) \| = 1$ from (II) above, which yields a contradiction with (\ref{ineq7}).
\end{proof}

\begin{remark}
It is worth noting that if $X$ has the CAP and the sequential Kadec-Klee property and $Y$ is uniformly convex, then the following are equivalent:
\begin{enumerate}
\itemsep0.3em 
\item[(a)] $\mathcal{P}(^N X, Y^*)$ is SSD.
\item[(b)] $(X, Y^*)$ has the $N$-homogeneous polynomial $\Loo$.
\item[(c)] $P$ strongly exposes $C(P)$ for every $P\in S_{\mathcal{P}(^N X, Y^*)}$.
\end{enumerate}
Indeed, implication (a)$\Rightarrow$(b) follows from Theorem~\ref{theoremA} (here we use the CAP and Kadec-Klee properties of the space $X$), while (b)$\Rightarrow$(c) follows from the first implication in the proof of Theorem~\ref{prop-star-equivalent-Loo}. Finally, implication (c)$\Rightarrow$(a) is the (almost trivial) Remark~\ref{remark strongly C(P)}. In view of Theorem~\ref{prop-star-equivalent-Loo}, we have that the equivalence (b)$\Leftrightarrow$(c) holds whenever $Y$ is uniformly convex and $X$ is uniformly convex or has the CAP and the sequential Kadec-Klee property. Since there exist uniformly convex spaces failing the CAP (see, for instance, \cite{Sza}), Theorem~\ref{prop-star-equivalent-Loo} apply to Banach spaces which are not in the hypotheses of Theorem~\ref{theoremA}.
\end{remark}

\subsection{Diagram with implications}\label{diagram}
We now provide all the properties that were discussed in the previous subsections, together with diagrams that show the connections between them and the hypotheses required for each connection.
Let $X, Y, X_1, \ldots, X_N$ be reflexive Banach spaces, and consider the following statements:

\begin{enumerate}
\itemsep0.3em 
  \item[(A)] $\mathcal{P}(^N X, Y^*)$ is SSD.
  \item[(B)] The pair $\left( \left( \sten X \right) \pten Y, \K \right)$ has the {\bf L}$_{o,o}$ (for linear functionals).
  \item[(C)] $\mathcal{P}(^N X, Y^*)$ is reflexive.
  \item[(D)] $\mathcal{P}(^N X, Y^*)=\mathcal{P}_{wsc}(^N X, Y^*)$.
  \item[(E)] The pair $(X, Y^*)$ has the $N$-homogeneous polynomial $\Loo$.
	\item[(F)] $P$ strongly exposes $C(P)$ for every $P\in S_{\mathcal{P}(^N X, Y^*)}$.
\end{enumerate}

and

\begin{enumerate}
\itemsep0.3em 
\item[(A')] The norm of $\mathcal{L}(X_1\times\cdots\times X_N)=\mathcal{L}(X_1\times\cdots\times X_{N-1},X_N^*)$ is SSD.
\item[(B')] The pair $\left( X_1\pten\cdots \pten X_N , \K \right)$ has the {\bf L}$_{o,o}$ (for linear functionals).
\item[(C')] $\mathcal{L}(X_1\times\cdots\times X_N)$ is reflexive.
\item[(D')] $\mathcal{L}(X_1\times\cdots\times X_{N-1}, X_N^*)=\mathcal{L}_{wsc}(X_1\times\cdots\times X_{N-1}, X_N^*)$.
\item[(E')] The pair $(X_1\times \cdots\times X_N, \K)$ has the $\Loo$ (for multilinear forms).
\end{enumerate}

Then the following implications hold:
	
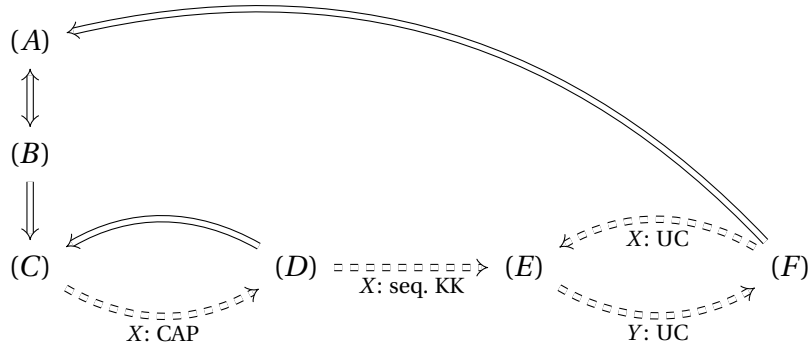
\begin{center}
		\begin{tikzpicture}[scale=1, baseline=0cm]
		\tikzstyle{caixa} = [rectangle, rounded corners, minimum width=.7cm, minimum height=.7cm, text centered, draw=black]
		\tikzstyle{dfletxa} = [double equal sign distance,-implies, shorten >= 2pt, , shorten <= 2pt]
		\tikzstyle{dfletxalr} = [double equal sign distance, implies-implies, shorten >= 2pt, , shorten <= 2pt]
		\tikzstyle{dfletxadotted} = [double equal sign distance, dashed, -implies, shorten >= 2pt, , shorten <= 2pt]
		
		\node (A) at (-7,-1.5) {$(A)$};
		\node (B) at (-7,-3) {$(B)$};
		\node (C) at (-7,-4.5) {$(C)$};
		\node (D) at (-3.5,-4.5) {$(D)$};
		\node (E) at (-0.5,-4.5) {$(E)$};
		\node (F) at (3,-4.5) {$(F)$};
		
		\draw [dfletxalr] (A) -- (B);
		\draw [dfletxa] (B) -- (C);
		%
		
		
		
		\draw[dfletxadotted, ->, >=implies, rounded corners] (C) to  [bend right]  node[below] {$\scriptstyle X: \ \text{CAP}$} (D);
		\draw[dfletxa, ->, >=implies, rounded corners] (D) to  [bend right]  node[below] {} (C);
		\draw[dfletxadotted, ->, >=implies] (D) to   node[below] {$\scriptstyle X: \ \text{seq. KK}$} (E);
		\draw[dfletxadotted, ->, >=implies, rounded corners] (E) to  [bend right]  node[below] {$\scriptstyle Y: \ \text{UC}$} (F);
		\draw[dfletxadotted, ->, >=implies, rounded corners] (F) to  [bend right]  node[below] {$\scriptstyle X: \ \text{UC}$} (E);
		\draw[dfletxa, ->, >=implies, rounded corners] (F) to  [bend right=30]  node[below] {} (A);
		
		
		\end{tikzpicture}	
	\end{center}

and

	
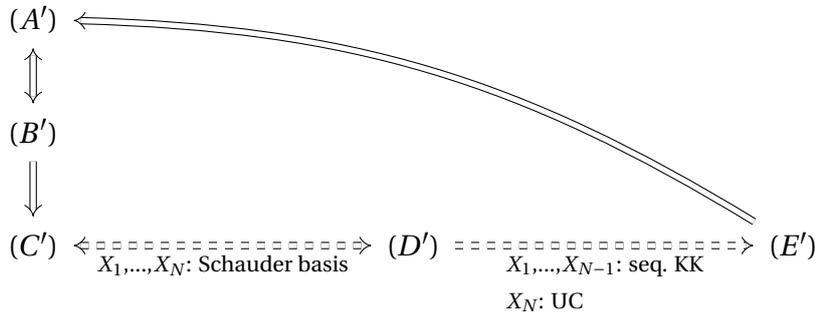
\begin{center}
		\begin{tikzpicture}[scale=1, baseline=0cm]
		\tikzstyle{caixa} = [rectangle, rounded corners, minimum width=.7cm, minimum height=.7cm, text centered, draw=black]
		\tikzstyle{dfletxa} = [double equal sign distance,-implies, shorten >= 2pt, , shorten <= 2pt]
		\tikzstyle{dfletxalr} = [double equal sign distance, implies-implies, shorten >= 2pt, , shorten <= 2pt]
		\tikzstyle{dfletxadotted} = [double equal sign distance, dashed, -implies, shorten >= 2pt, , shorten <= 2pt]
		
		\node (A') at (-7,-1.5) {$(A')$};
		\node (B') at (-7,-3) {$(B')$};
		\node (C') at (-7,-4.5) {$(C')$};
		\node (D') at (-2,-4.5) {$(D')$};
		\node (E') at (3,-4.5) {$(E')$};
		
		\draw [dfletxalr] (A') -- (B');
		\draw [dfletxa] (B') -- (C');
	  \draw [dfletxadotted, <->, >=implies,] (C') to node[below] {$\scriptstyle X_1,\ldots, X_N: \ \text{Schauder basis}$} (D'); 
		\draw[dfletxadotted, ->, >=implies, align=left] (D') to node[below] {$\scriptstyle X_1, \ldots, X_{N-1}: \ \text{seq. KK}$ \\ $\scriptstyle X_N: \ \text{UC}$} (E');
		\draw[dfletxa, ->, >=implies, rounded corners] (E') to [bend right=15] node[below] {} (A');
		
		
		\end{tikzpicture}	
	\end{center}

\section{On the (uniform) strong subdifferentiability of $X \pten Y$ and $\sten X$}
\label{Section:TheoremC}

Analogously of what we do in Section \ref{Section:TheoremAandB}, where we use the \emph{$N$-homogeneous polynomial} (respectively, \emph{multilinear}) $\Loo$ as a tool to obtain the strong subdifferentiability of many spaces of $N$-homogeneous polynomials (respectively, multilinear mappings), in this section we establish a connection between strong subdifferentiability of (symmetric) projective tensor products and \emph{Bishop-Phelps-Bollob\'as point type properties} which, roughly speaking, are the dual counterpart of $\Loo$ properties. Let us briefly clarify the different \emph{point properties} we will be dealing with throughout the section. In \cite{DKL} (see also \cite{DKKLM1}) the authors defined and studied the \emph{Bishop-Phelps-Bollobás point property} (BPBpp, for short) for linear and bilinear operators. We state this property in the next definition, and extend it to the polynomial setting.
\begin{definition}
Let $N\in \mathbb{N}$ and $X, X_1, \dots, X_N, Y$ be Banach spaces. We say that the pair \linebreak $(X_1\times \cdots \times X_N, Y)$ has the \emph{BPBpp} if given $\e > 0$, there exists $\eta(\e) > 0$ such that whenever $A \in \mathcal{L} (X_1\times \cdots \times X_N, Y)$ with $\|A\| = 1$ and $(x_1, \ldots, x_N)\in S_{X_1}\times \cdots \times S_{X_N}$ satisfy $$\|A(x_1, \ldots, x_N) \| > 1 - \eta(\e),$$ there is a new $N$-linear mapping $B \in \mathcal{L} (X_1\times \cdots \times X_N, Y)$ with $\|B\| = 1$ such that 
$$
\|B(x_1, \ldots, x_N)\| = 1 \quad \text{and}\quad \|B - A\| < \e.
$$ 
Analogously, we say that the pair $(X, Y)$ has the \emph{$N$-homogeneous polynomial BPBpp} if given $\e > 0$, there exists $\eta(\e) > 0$ such that whenever $P\in \mathcal{P}(^NX, Y)$ with $\|P\|=1$ and $x\in S_X$ satisfy \linebreak $\|P(x)\|>1-\eta(\e)$, there exists $Q\in \mathcal{P}(^NX, Y)$ with $\|Q\|=1$ such that $\|Q(x)\|=1$ and $\|P-Q\|<\e$.
\end{definition}

Note that these properties are the \emph{uniform} versions of properties $\Lpp$ from Definition~\ref{def Lpp and Loo}, in the sense that the $\eta$ does not depend on the points $(x_1,\dots, x_N)\in S_{X_1}\times \cdots \times S_{X_N}$ and $x\in S_X$ but only on $\e > 0$.
The reason to consider all these \emph{Bishop-Phelps-Bollob\'as point type properties} is that we will derive some strong subdifferentiability results for the Banach spaces $\ell_2\pten \stackrel{N}{\cdots} \pten \ell_2$, $\sten \ell_2$, $c_0\pten c_0$ and $c_0 \widehat{\otimes}_{\pi_s} c_0$ from BPBpp and $\Lpp$ type results (see Subsection~\ref{the proof of thm C} below). 

It is not difficult to see that, both in the above definition and in Definition~\ref{def Lpp and Loo} (i), we can take $\|A\|$ and $\|P\|$ less than or equal to one (not necessarily $\|A\|=\|P\|=1$) by making a standard change of parameters. We will make use of this fact without any explicit mention.

\subsection{The tools} We start this section by proving the first tool we need to get the results from Theorem~\ref{theoremC}. It is known that the pair $(X, \K)$ has the BPBpp for linear functionals if and only if $X$ is uniformly smooth (see \cite[Propositon 2.1]{DKL}). Note that the uniform smoothness of $X$ is equivalent to say that the norm of $X$ is USSD on $U=S_X$ (recall Definition~\ref{definition:uniform-SSD}). Our next result is a \emph{localization} of the above mentioned characterization.

\begin{proposition}\label{thm BPBpp for a set is unif SSD}
	Let $X$ be a Banach space and $U\subseteq S_{X}$. Then the following are equivalent.
	\begin{enumerate}
			\itemsep0.3em 
	  \item[(a)] The norm of $X$ is USSD on $U$. 
		\item[(b)] The pair $(X, \K)$ has the \emph{BPBpp for the set $U$}, that is, given $\e > 0$, there exists $\eta(\e) > 0$ such that whenever $x_1^*\in S_{X^*}$ and $u\in U$ satisfy $|x_1^*(u)|>1-\eta(\e)$, there exists $x_2^*\in S_{X^*}$ such that $|x_2^*(u)|=1$ and $\|x_1^*-x_2^*\|<\e$.
	\end{enumerate}
\end{proposition} 

\begin{proof} Suppose that the norm of $X$ is USSD on the set $U$ and let $\delta>0$ be such that, if $0<t<\delta$, then
 $$\frac{\Vert u + tz\Vert -1}{t}- \tau(u,z) < \frac{\varepsilon}{2}$$
for every $(u,z)\in U\times B_{X}$ (recall (\ref{eq:SSD3})). We will show that the pair $(X, \K)$ has the BPBpp for the set $U$ with $\eta(\varepsilon) := \frac{\delta \varepsilon}{4} >0$. Suppose that this is not the case. Then, there exist $u\in U$ and $x^* \in S_{X^*}$ such that 
$\operatorname{Re} x^*(u) > 1- \eta(\varepsilon)$
and $\Vert x^* -\tilde{x}^* \Vert > \varepsilon$ for every $\tilde{x}^* \in S_{X^*}$ satisfying $\tilde{x}^*(u)=1$. 
Then, $D(u)$ and $x^*+\varepsilon B_{X^*}$ are $w^*$-compact, convex, and disjoint sets. Now, by the Hahn-Banach separation theorem there exists $z \in S_X$ such that 
$$\tau(u,z)=\max\{\operatorname{Re} \tilde{x}^* (z): \tilde{x}^* \in D(u) \}\leq \min \{ \operatorname{Re} (x^*+\varepsilon z^*)(z): z^*\in B_{X^*} \}=\operatorname{Re}x^*(z)-\varepsilon.$$
Then, for $t=\frac{\delta}{2}$ we have
\begin{eqnarray*}
\frac{\varepsilon}{2} &>& \frac{\Vert u + tz \Vert -1}{t}- \tau(u,z) \\
&\geq& \frac{\operatorname{Re} x^* (u+tz) -1}{t}-\operatorname{Re}x^*(z) + \varepsilon \\
&=& \frac{\operatorname{Re}x^*(u) -1}{t}+ \varepsilon \\
&\geq& \frac{1-\eta-1}{t}+ \varepsilon = \frac{-\eta}{t}+ \varepsilon = \frac{- \delta \varepsilon}{4} \frac{2}{\delta}+ \varepsilon= \frac{\varepsilon}{2}\
\end{eqnarray*}
which is a contradiction. The other implication is analogous to \cite[Proposition 2.1]{DKL}.
\end{proof}

Although Proposition~\ref{thm BPBpp for a set is unif SSD} may seem artificial at a first glance, it is useful to relate the BPBpp (for multilinear operators and homogeneous polynomials) with the geometry of the (symmetric) tensor products. From now on, given $X, X_1,\ldots, X_N$ Banach spaces, we denote
\begin{equation*} 
U:=\Big\{x_1\otimes\cdots\otimes x_N:\,\, \|x_j\|=1 \Big\} \subseteq S_{X_1\hat{\otimes}_\pi\cdots \hat{\otimes}_\pi X_N} \quad\text{and}\quad U_s:= \Big\{\otimes^N x:\,\, \|x\|=1 \Big\}\subseteq S_{\sten X}.
\end{equation*} 

\begin{proposition}\label{thm BPBpp for bilinear and polynomial}
Let $X, X_1,\ldots, X_N$ be Banach spaces.
\begin{enumerate}
		\itemsep0.3em 
  \item[(i)] $X_1\hat{\otimes}_\pi\cdots \hat{\otimes}_\pi X_N$ is USSD on $U$ if and only if $(X_1\times \cdots\times X_N, \K)$ has the BPBpp.
	\item[(ii)] $\sten X$ is USSD on $U_s$ if and only if $(X, \K)$ has the $N$-homogeneous polynomial BPBpp.
\end{enumerate}	
\end{proposition}

\begin{proof}
	A simple linearizing argument shows that $(X_1\times \cdots\times X_N, \K)$ has the BPBpp if and only if $(X_1\hat{\otimes}_\pi\cdots \hat{\otimes}_\pi X_N,\K)$ has the BPBpp for the set $U$ (recall the definition in item (b) of Proposition~\ref{thm BPBpp for a set is unif SSD}). Then the statement follows from Proposition~\ref{thm BPBpp for a set is unif SSD}. A similar argument can be applied in the polynomial context.
\end{proof}

When dealing with (non-necessarily uniform) strong subdifferentiability of tensor products, we have the analogous \emph{local} version of Proposition~\ref{thm BPBpp for bilinear and polynomial}. Recall the definition of the $N$-homogeneous polynomial $\Lpp$ in Definition~\ref{def Lpp and Loo}.

\begin{proposition} \label{local} Let $X, X_1, \ldots, X_N$ be Banach spaces.
	\begin{enumerate}
			\itemsep0.3em 
		\item[(i)] $X_1\hat{\otimes}_\pi\cdots \hat{\otimes}_\pi X_N$ is SSD on $U$ if and only if $(X_1\times \cdots \times X_N, \K)$ has the $\Lpp$.
		\item[(ii)] $\sten X$ is SSD on $U_s$ if and only if $(X, \K)$ has the $N$-homogeneous polynomial $\Lpp$.
	\end{enumerate}
\end{proposition}
\begin{proof}
	The proofs of item (i) and (ii) are analogous, hence we only prove (i).
	The pair \linebreak $(X_1\times \cdots\times X_N, \K)$ fails the $\Lpp$ if and only if there is $(x_1,\ldots,x_N)\in S_{X_1}\times\cdots \times S_{X_N}$ and a sequence of norm-one $N$-linear forms $L_n:X_1\times\cdots \times X_N\rightarrow \K$, such that 
	\begin{equation*}
	L_n(x_1,\ldots,x_N)\rightarrow 1 \quad \text{and} \quad \operatorname{dist}(L_n, D(x_1,\dots, x_N)) \not\rightarrow 0,
	\end{equation*}
	where $D(x_1,\dots, x_N) = \{L \in \mathcal{L}(X_1\times\cdots \times X_N,\K): L(x_1,\ldots,x_N)=\Vert L\Vert=1\}$. In terms of projective tensor products, this is equivalent to say that there is an element $u=x_1\otimes\cdots\otimes x_N \in U$ and a sequence of norm-one linear functionals $\varphi_n\in (X_1\hat{\otimes}_\pi\cdots \hat{\otimes}_\pi X_N)^*$ (each $\varphi_n$ is the functional associated to $L_n$) such that
	\begin{equation*}
	\varphi_n(u)\rightarrow 1 \quad \text{and} \quad \operatorname{dist}(\varphi_n, D(u)) \not\rightarrow 0.
	\end{equation*} 
	By \cite[Theorem~1.2]{FP}, this is equivalent to the norm of $X_1\hat{\otimes}_\pi\cdots \hat{\otimes}_\pi X_N$ not being SSD at $u$.
\end{proof}

\subsection{The proof of Theorem~\ref{theoremC}}\label{the proof of thm C}

Now we are ready to walk towards the proof of Theorem \ref{theoremC}. 
On the one hand, in Theorem~\ref{theorem:micro} we will prove that if $X, X_1,\dots, X_N$ are Banach spaces with micro-transitive norms (see the definition in the paragraph below) then $(X_1\times\cdots\times X_N, \K)$ and $(X,\K)$ have the multilinear and $N$-homogeneous polynomial BPBpp, respectively. Since Hilbert spaces have micro-transitive norms, this result together with Proposition~\ref{thm BPBpp for bilinear and polynomial} give items (i) and (iii) of Theorem~\ref{theoremC}. On the other hand, we will see that the pair $(c_0, \C)$ has the $2$-homogeneous $\Lpp$ in the complex case. In fact, we will prove a slightly stronger result with codomain a finite dimensional Hilbert space. This, together with Proposition~\ref{local}, prove item (ii). The item (iv) follows, analogously, from Proposition~\ref{local} and a result in \cite{ChoKim}, which states that 
$(c_0\times c_0, \C)$ has the bilinear $\Lpp$ in the complex case. Finally, item (v) follows from the fact that $\ell_1^N\pten Y=\ell_1^N(Y)$ and that $Y\oplus_1 Y$ is SSD if and only if $Y$ is SSD (see, for instance, \cite[Proposition~2.2]{FP}).

\subsubsection{The BPBpp on spaces with micro-transitive norms}
Given a Hausdorff topological group $G$ with identity $e$ and a Hausdorff space $T$, we say an action $G \times T \rightarrow T$ is micro-transitive if for every $x \in T$ and every neighborhood $U$ of $e$ in $G$, the orbit $Ux$ is a neighborhood of $x$ in $T$. In terms of Banach spaces, we say that the norm of a Banach space is \emph{micro-transitive} if its group of surjective isometries acts micro-transitively on its unit sphere. Equivalently, we have that the norm of a Banach space is micro-transitive if and only if there is a function $\beta : (0,2) \rightarrow \mathbb{R}^+$ such that if $x,y \in S_X$ satisfy $\|x - y\| < \beta(\eps)$, then there is a surjective isometry $T \in \mathcal{L} (X, X)$ satisfying $T(x) =y$ and $\|T-\id \| < \e$ (see \cite[Proposition 2.1]{CDKKLM}).

\begin{theorem} \label{theorem:micro}
	Let $X, X_1,\dots, X_N$ be Banach spaces with micro-transitive norms and $Z$ an arbitrary Banach space. Then the following results hold. 
	\begin{itemize}
	\itemsep0.3em
		\item[(i)] The pair $(X_1\times\cdots\times X_N, Z)$ has the BPBpp.
		\item[(ii)] The pair $(X, Z)$ has the $N$-homogeneous polynomial BPBpp.
	\end{itemize}
	
\end{theorem}

\begin{proof} Let us first prove item (i). For simplicity, we prove the case $N=2$. By \cite[Collorary 2.13]{CDKKLM}, $X_1$ and $X_2$ are uniformly convex (also, uniformly smooth). Then, it follows from \linebreak \cite[Theorem 2.2]{ABGM} that $(X_1 \times X_2, Z)$ has the Bishop-Phelps-Bollob\'{a}s property with $\eps \mapsto \eta(\eps)$. Let $A \in \mathcal{L}(X_1\times X_2, Z)$ with $\|A \| =1$ and $\|A (x_0, y_0 )\| > 1 -\eta'(\eps)$ for some $x_0 \in S_{X_1}$ and $y_0 \in S_{X_2}$, where 
	\begin{equation*}
	\eta'(\eps)= \eta\left( \min\left\{ \frac{\eps}{3}, \beta_{X_1} \left(\frac{\eps}{3}\right), \beta_{X_2} \left(\frac{\eps}{3}\right) \right\} \right).
	\end{equation*}
	Here, $\beta_{X_1}, \beta_{X_2} : (0,2) \rightarrow \R^+$ are functions induced from the micro-transitivity of the norms $X_1$ and $X_2$, respectively. Then there are $\widetilde{B} \in \mathcal{L}(X_1 \times X_2, Z)$ and $(\tilde{x_0}, \tilde{y_0}) \in S_{X_1} \times S_{X_2}$ such that 
	\begin{enumerate}
		\itemsep0.3em 
		\item $\| \widetilde{B} \| = \| \widetilde{B} (\tilde{x_0}, \tilde{y_0}) \| = 1,$ 
		\item $\max\{\|x_0-\tilde{x_0}\|, \|y_0-\tilde{y_0}\|\} < \min\left\{ \beta_{X_1} \left(\frac{\eps}{3}\right), \beta_{X_2} \left(\frac{\eps}{3}\right) \right\},$
		\item $\| \widetilde{B} - A \| < \frac{\eps}{3}$.
	\end{enumerate} 
	Let $T_1 \in \mathcal{L}(X_1, X_1)$ and $T_2 \in \mathcal{L}(X_2, X_2)$ be surjective isometries such that
	$$
	T_1(x_0) = \tilde{x_0} \quad \text{and} \quad \|T_1 -\id_{X_1} \| < \frac{\e}{3},
	$$
and that 
$$
T_2(y_0) = \tilde{y_0}\quad \text{and} \quad\|T_2 -\id_{X_2} \| < \frac{\e}{3}.
$$
Define $B(x,y) = \widetilde{B} (T_1(x), T_2(y))$ for every $(x, y) \in X_1 \times X_2$. Clearly, $B \in \mathcal{L}(X_1\times X_2, Z)$ and $\|B \| \leq 1$. Note that 
	\begin{equation*}
	\|B(x_0, y_0)\| = \|\widetilde{B} (T_1(x_0), T_2(y_0))\| = \|\widetilde{B} (\tilde{x_0}, \tilde{y_0}) \| = 1
	\end{equation*} 
	and, hence, $\|B\| = 1$. Also, 
	\begin{align*}
	\|A(x,y) - B(x,y) \| &\leq \| A(x,y) - \widetilde{B}(x, y) \| + \| \widetilde{B}(x, y) - \widetilde{B}(x, T_2(y))\| + \| \widetilde{B}(x, T_2(y)) - \widetilde{B}(T_1(x), T_2(y))\| \\
	&< \frac{\eps}{3} + \| y - T_2(y)\| + \| x - T_1(x) \| \\
	&< \eps 
	\end{align*} 
	for every $(x,y) \in S_{X_1}\times S_{X_2}$. This shows that $\|A - B\| \leq \eps$ and completes the proof. 
	
	The proof of (ii) follows the same line of the previous one. Using again the fact that a Banach space $X$ with micro-transitive norm is uniformly convex and taking \cite[Theorem~3.1]{Acoet6} into account, we deduce that $(X,Z)$ has the Bishop-Phelps-Bollob\'as property. Hence, if \linebreak $P \in \mathcal{P}(^NX, Z)$ with $\|P \| =1$ is such that $$\|P(x_0)\| > 1 - \eta\left( \min\left\{ \frac{\eps}{N+1}, \beta_{X} \left(\frac{\eps}{N+1}\right) \right\} \right)$$ for some $x_0 \in S_{X}$, then there exist $\widetilde{Q}\in \mathcal{P}(^NX, Z)$, $\|\widetilde{Q}\|=1$, and $\tilde{x_0}\in S_X$ such that
	$$
	\|\widetilde{Q}(\tilde{x_0})\|=1, \quad \|x_0-\tilde{x_0}\|<\frac{\e}{N+1}, \quad \text{and}\quad \|P-\widetilde{Q}\|<\frac{\e}{N+1}.
	$$
	Finally, letting $T\in \mathcal{L}(X,X)$ be a surjective isometry such that $T(x_0)=\tilde{x_0}$ and \linebreak $\|T-\id_{X}\|<\frac{\e}{N+1}$, we consider $Q(x)=\widetilde{Q}(T(x))$ which attain its norm at $x_0\in S_X$ and approximates $P$. The details are left to the reader.
\end{proof}

The previous theorem together with Proposition~\ref{thm BPBpp for bilinear and polynomial} shows that if $X, X_1,\dots, X_N$ are Banach spaces with micro-transitive norms, then $X_1\hat{\otimes}_\pi\cdots \hat{\otimes}_\pi X_N$ is USSD on $U$ and $\sten X$ is USSD on $U_s$. It is an open problem the existence of Banach spaces, other than Hilbert spaces, having micro-transitive norms.

\subsubsection{The $\Lpp$ for $2$-homogeneous polynomials on $c_0$} 
In order to prove the $2$-homogeneous polynomial $\Lpp$ for the pair $(c_0, H)$, with $H$ a finite dimensional complex Hilbert space, we need the following key result which is motivated by \cite[Proposition 2]{ABC}. For a subset $A$ of $\mathbb{N}$, let us denote by $P_A$ the natural projection from $c_0$ onto $\ell_\infty^A$. 

\begin{proposition}\label{prop:ABC}
Consider the complex space $c_0$ and a complex Hilbert space $H$. Given \linebreak$x_0\in S_{c_0}$ and $\eps >0$,  there exists $\eta(\eps, x_0)>0$ such that $\|P - P\circ P_A\| < \eps$  for any  $P \in \mathcal{P}(^2 c_0, H)$ with $\| P \|=1$ and $\|P(x_0)\| > 1-\eta(\eps,x_0)$, where $A:=\{i\in\N:\,\, |x_0(i)|=1\}$.
\end{proposition}

\begin{proof}
Suppose that $\|P(x_0)\| > 1 -\eps_0$. We will see that $\eps_0 >0$ can be chosen depending on $x_0$ and $\eps>0$ in such a way that $\| P - P\circ P_A\| < \eps$. 
For simplicity, and without loss of generality, we will suppose that $A=\{1,\dots,n\}$. Now, consider 
$$
y=(0,\dots,0, y_{n+1}, y_{n+2}, \dots)\in B_{c_0}.
$$
Then, for every $\lambda\in \C$, $|\lambda|= 1- \max \{ |x_0 (i)| : i > n\}$, we have
\begin{equation}\label{eq pol+norm one}
\| P(x_0)\pm 2\check{P}(x_0, \lambda y) + \lambda^2 P(y) \| = \|P(x_0\pm \lambda y)\|\leq 1.
\end{equation}
Then,
$$
\| P(x_0)+\lambda^2 P(y) \| \leq 1.
$$
Note that 
$$
\| P(x_0)\|^2 + 2 \text{Re} \langle P(x_0), \lambda^2 P(y) \rangle + |\lambda|^4 \|P(y)\|^2 = \| P(x_0)+\lambda^2 P(y) \|^2 \leq 1,
$$
where $\langle \cdot, \cdot \rangle$ is the inner product on $H$. By choosing $\lambda$ so that $\langle P(x_0), \lambda^2 P(y) \rangle$ is purely imaginary we deduce
$$
\left(\| P(x_0)\| ^2 + |\lambda|^4 \| P(y)\|^2\right)^{\frac{1}{2}} \leq 1
$$
Since $\| P(x_0) \| >1-\e_0$, we have
$$
\left((1-\e_0)^2 +|\lambda|^4 \| P(y)\|^2 \right)^{\frac{1}{2}} \leq 1
$$
and, consequently, 
\begin{equation}\label{estimation A comp}
\|P(y)\|\leq \frac{\sqrt{2\e_0-\e_0^2}}{|\lambda|^2} =: \beta_1 (\e_0). 
\end{equation}
Now, returning to \eqref{eq pol+norm one}, 
$$
\|P(x_0) + 2\lambda \check{P}(x_0, y)\| - \|\lambda^2 P(y)\| \leq \|P(x_0)\pm 2\check{P}(x_0, \lambda y) + \lambda^2 P(y)\| \leq 1.
$$
Then, from the estimate \eqref{estimation A comp}, we deduce that
$$
\|P(x_0)+2\lambda \check{P}(x_0,y)\| \leq 1+ |\lambda|^2 \beta_1(\e_0). 
$$
Note again that 
$$
\| P(x_0)\|^2 + 2\text{Re} \langle P(x_0), 2\lambda \check{P}(x_0, y)\rangle + 4|\lambda|^2 \|\check{P}(x_0, y)\|^2  \leq (1+ |\lambda|^2 \beta_1(\e_0))^2. 
$$
Choosing $\lambda$ so that $\langle P(x_0), 2\lambda \check{P}(x_0, y)\rangle$ is purely imaginary we have
$$
\|P(x_0)\|^2 + 4|\lambda|^2 \| \check{P}(x_0,y) \|^2 \leq \left( 1 + |\lambda|^2 \beta_1(\e_0)\right)^2,
$$
from where we deduce (using again the fact that $\|P(x_0)\|>1-\e_0$) that
\begin{equation}\label{eq beta2}
\|\check{P}(x_0,y)\|< \frac{1}{2|\lambda|} \left( \left( 1 + |\lambda|^2 \beta_1 (\e_0)\right)^2 - (1-\e_0)^2 \right)^{\frac{1}{2}} =: \beta_2 (\e_0).
\end{equation}

Given any $x\in c_0$, let us call $x^A=P_A(x)=(x_1,\dots, x_n, 0,\dots)$ and $x^{A^c}=(0,\dots,0, x_{n+1}, x_{n+2},\dots)$.
Note that
$$
P(x_0^A)=P(x_0-x_0^{A^c})=P(x_0)-2\check{P}(x_0, x_0^{A^c})+P(x_0^{A^c})
$$
and, hence,
$$
1-\e_0 < \|P(x_0)\|=\|P(x_0^A)+2\check{P}(x_0, x_0^{A^c})-P(x_0^{A^c})\|<\|P(x_0^A)\| + 2\beta_2 (\eps_0) + \beta_1 (\e_0).
$$
This gives $\|P(x_0^A)\|>1-\e_0 - 2\beta_2(\e_0) - \beta_1 (\eps_0) =: 1-\beta_3 (\eps_0)$. In particular, this shows that \linebreak $\| P \circ P_A \| > 1 - \beta_3 (\eps_0)$. 
Using the finite dimensionality of $\ell_\infty^A$, we take $u =(u_1,\ldots, u_n, 0,0,\ldots)$ an element of $S_{c_0}$ such that $\| (P \circ P_A)(u)\| = \| P \circ P_A\|$. Applying the maximum modulus principle, we may assume that $|u_1|=\cdots=|u_n|=1$. Moreover, by a simple change of variables we may assume $u=(1,\ldots,1,0,0,\ldots)$. 

Using the fact that $\|P(u)\| > 1-\beta_3(\eps_0)$ and arguing as in \eqref{eq beta2}, we can prove that \linebreak $\|\check{P}(u^A, y)\|<\beta_4 (\e_0)$ for every $y=(0,\dots,0, y_{n+1}, \dots)\in B_{c_0}$ for some $\beta_4 (\eps_0) >0$ satisfying that $\beta_4(\eps_0) \rightarrow 0$ as $\eps_0 \rightarrow 0$.

Now, let us consider the basis of $\C^n$,
\begin{eqnarray*}
z_1&=&(1,\dots, 1)\\
z_2&=&(1, -n+1, 1,\dots, 1)\\
&\vdots&\\
z_n&=&(1,\dots,1, -n+1)
\end{eqnarray*}
and $\overline{z_j}=(z_j, 0,\dots)\in c_0$. Given $(x_1,\dots, x_n)\in \C^n$ it can be checked that
\begin{equation}\label{basis}
(x_1,\dots, x_n)=\frac{1}{n}(x_1+\cdots+x_n)z_1 + \frac{1}{n}\sum_{j=2}^n(x_1-x_j)z_j.
\end{equation}
Now, for any $(x_1,\dots,x_n)\in B_{\ell_\infty^n}$ and any $y=(0,\dots,0, y_{n+1},\dots)\in B_{c_0}$,
$$
P(x_1,\dots, x_n, y_{n+1},\dots)=P(x_1,\dots, x_n,0\dots)+ 2\check{P}((x_1,\dots, x_n,0\dots), y) + P(y).
$$
In virtue of \eqref{basis},
$$
\check{P}((x_1,\dots, x_n,0\dots), y) = \frac{1}{n}(x_1+\cdots+x_n) \check{P}(\overline{z_1}, y)+\frac{1}{n}\sum_{j=2}^n(x_1-x_j) \check{P}(\overline{z_j}, y).
$$
If we call $\psi_j(\cdot)=\frac{2}{n}\check{P}(\overline{z_j}, \cdot)$, then we have
\begin{equation}\label{eq psi}
P(x_1,\dots, x_n, y_{n+1},\dots)=P(x_1,\dots, x_n,0\dots) + (x_1+\cdots+x_n)\psi_1(y) + \sum_{j=2}^n(x_1-x_j)\psi_j(y) + P(y).
\end{equation}
There is a little abuse of notation: when we write $\psi_j(y)$, we should write $\psi_j(y_{n+1}, y_{n+2}, \dots)$ (however we will keep this notation).

Let us prove that $\|\psi_j\|$ is small for $j=2,\dots,n$. We will see the case $j=2$, being the others analogous. From equation \eqref{eq psi}, choosing $x_2=e^{i\theta}$ (for any real $\theta$) and $x_j=1$ if $j\neq 2$, we deduce
$$
P(1, e^{i\theta}, 1,\dots, 1, 0,\dots) + (1-e^{i\theta})\psi_2(y)=P(1, e^{i\theta}, 1,\dots, 1, y_{n+1},\dots) - (n-1+e^{i\theta})\psi_1(y)-P(y).
$$
Then,
$$
\| P(1, e^{i\theta}, 1,\dots, 1, 0,\dots) + (1-e^{i\theta})\psi_2(y) \| \leq 1 + n\|\psi_1(y)\|+\|P(y)\|< 1 + 2\beta_4 (\e_0)+\beta_1 (\e_0),
$$
that is,
\begin{align*}
\| P(1, e^{i\theta}, 1,\dots, 1, 0,\dots) \|^2 + 2 \text{Re} \langle  &P(1, e^{i\theta}, 1,\dots, 1, 0,\dots) , (1-e^{i\theta})\psi_2(y)\rangle \\
&\qquad + \| (1-e^{i\theta})\psi_2(y)\|^2 \leq (1 + 2\beta_4 (\e_0)+\beta_1 (\e_0))^2. 
\end{align*} 
As we can vary the argument of $y$ (independent of $\theta$), we deduce that
$$
\|P(1, e^{i\theta}, 1,\dots, 1, 0,\dots)\|^2 + \|(1-e^{i\theta})\psi_2(y)\|^2 \leq (1 + 2\beta_4 (\e_0)+\beta_1 (\e_0))^2
$$
and, since it holds for every $y\in B_{c_0}$, then
$$
\|P(1, e^{i\theta}, 1,\dots, 1, 0,\dots) \|^2 + |1-e^{i\theta}|^2 \,\|\psi_2\|^2 \leq (1 + 2\beta_4 (\e_0)+\beta_1 (\e_0))^2. 
$$
Then, 
\begin{align*}
|1-e^{i\theta}|^2 \,\|\psi_2\|^2 &\leq (1 + 2\beta_4 (\e_0)+\beta_1 (\e_0))^2 - \|P(1, e^{i\theta}, 1,\dots, 1, 0,\dots)\|^2 \\
&= \underbrace{1-\|P(1,\dots,1,0,\dots)\|^2}_{(I)} \\
&\qquad + \underbrace{\|P(1,\dots,1,0,\dots)\|^2 - \|P(1,e^{i\theta},1,\dots,1,0,\dots)\|^2}_{(II)} + \beta_5 (\eps_0).
\end{align*}
where $1+\beta_5 (\eps_0) := (1 + 2\beta_4 (\e_0)+\beta_1 (\e_0))^2$. 
Given that $\|P(1,\dots,1,0,\dots)\|=\| P \circ P_A \| >1-\beta_3 (\eps_0)$, then $(I)< 2\beta_3 (\eps_0) - \beta_3 (\eps_0)^2$.

Define $f(\theta)= \|P(1,\ldots, 1, 0,\ldots)\|^2 - \|P(1, e^{i\theta}, 1,\dots, 1, 0,\dots)\|^2$ and $g(\theta)=|1-e^{i\theta}|$ and note that $g(\theta)=2\sin(\theta/2)$ for $\theta\geq 0$. It is worth noting that $\theta \mapsto \|P(1, e^{i\theta}, 1,\dots, 1, 0,\dots)\|$ is differentiable at $\theta=0$ (because $P$ is holomorphic and $P(1, e^{i\theta}, 1,\dots, 1, 0,\dots) \not\to 0$ when $\theta \to 0$).
Note from L'H$\hat{\text{o}}$pital's rule that 
\begin{equation}\label{limit lhopital}
\lim_{\theta\to 0^+} \frac{f(\theta)}{g(\theta)^2}=  \lim_{\theta\to 0^+} \frac{-2 \|P(1, e^{i\theta}, 1,\dots, 1, 0,\dots)\| \big(\frac{d}{d\theta} \|P(1, e^{i\theta}, 1,\dots, 1, 0,\dots)\| \big)}{4\sin(\theta/2)\cos(\theta/2)}.
\end{equation}
On the one hand it is clear that 
$$
\lim_{\theta\to 0^+} \frac{\frac{d}{d\theta} \|P(1, e^{i\theta}, 1,\dots, 1, 0,\dots)\|}{\cos(\theta/2)}=0
$$
since $ \|P(1, e^{i\theta}, 1,\dots, 1, 0,\dots)\|$ has a local maximum at $\theta=0$ and, consequently, $$\frac{d}{d\theta} \|P(1, e^{i\theta}, 1,\dots, 1, 0,\dots)\|_{\big|_{\theta=0}}=0.$$ On the other hand, applying again  L'H$\hat{\text{o}}$pital's rule we have
$$
\lim_{\theta\to 0^+} \frac{\|P(1, e^{i\theta}, 1,\dots, 1, 0,\dots)\|}{\sin(\theta/2)}=\lim_{\theta\to 0^+} \frac{2 \frac{d}{d\theta} \|P(1, e^{i\theta}, 1,\dots, 1, 0,\dots)\|}{\cos(\theta/2)}=0
$$
Then, going back to \eqref{limit lhopital} we obtained that
\begin{equation}\label{limit}
\lim_{\theta\to 0^+} \frac{f(\theta)}{g(\theta)^2}=0.
\end{equation}
Take $\theta_0 >0$ such that $|1-e^{i\theta_0}|^2 = 4 \sin^2 (\theta_0/2) = \gamma(\eps_0) := (2\beta_3 (\eps_0) - \beta_3(\eps_0)^2 + \beta_5 (\eps_0))^{\frac{1}{2}}$. 
Then we obtain
$$
\gamma (\e_0) \|\psi_2\|^2\leq \gamma(\eps_0)^2 + f(2 \arcsin (\gamma(\eps_0)^{1/2}/2)) 
$$
from where we deduce that
$$
\|\psi_2\|^2\leq \gamma(\eps_0) + \frac{f(2 \arcsin (\gamma(\eps_0)^{1/2}/2)) }{\gamma(\eps_0)}\xrightarrow[\eps_0 \to 0]{} 0.
$$
The limit 
$$
 \frac{f(2 \arcsin (\gamma(\eps_0)^{1/2}/2)) }{\gamma(\eps_0)}\xrightarrow[\eps_0 \to 0]{} 0
$$
follows from \eqref{limit} and the fact that 
$$
 \frac{f(2 \arcsin (\gamma(\eps_0)^{1/2}/2)) }{\gamma(\eps_0)}=\frac{f(\theta_0)}{g(\theta_0)^2}
$$
and $\theta_0$ goes to $0$ as $\eps_0 \to 0$.
As we already mentioned, we can obtain the same bounds for $\psi_3,\dots, \psi_n$. Then, looking at \eqref{eq psi}, we conclude that there is some $\eps_0 = \eta(\eps, x_0) >0$ such that 
$$
\|P(x_1,\dots, x_n, y_{n+1},\dots)-P(x_1,\dots, x_n,0\dots)\|< \eps.
$$
for every $(x_1,\dots, x_n)\in B_{\ell_\infty^n}$ and $y=(0,\dots,0,y_{n+1}, \dots)\in B_{c_0}$. This proves the statement.
\end{proof}

Now we are ready to state and prove the main result of this subsection. Before that, let us simply note that if $X, Y$ are finite dimensional Banach spaces, then the pair $(X, Y)$ has the $2$-homogeneous polynomial $\Lpp$. Indeed, it follows by contradiction using the compactness of the unit ball  $B_{\mathcal{P}(^2X, Y)}$.

\begin{theorem} \label{c0-result}
For a finite dimensional Hilbert space $H$, the pair $(c_0, H)$ has the $2$-homoge\-neous polynomial $\Lpp$ in the complex case. 
\end{theorem}

\begin{proof}
Let $\eps>0$ and $x_0 \in S_{c_0}$ be fixed. Consider the finite set $A :=\{i \in \mathbb{N}:|x_0(i)|=1\}$. Since $\ell_\infty^A$ and $H$ are finite dimensional, we can consider $\tilde{\eta}(\eps, x_0^A) >0$ from the $2$-homogeneous polynomial $\Lpp$ for the pair $(\ell_\infty^A, H)$. We may assume that $(1-\e)\tilde{\eta}(\eps, x_0^A)<\e$. Suppose that 
\[
\|P(x_0)\| > 1 - \min \left\{\eta \left( \min \left\{ \frac{(1-\eps)\tilde{\eta} (\eps,x_0^A)}{2}, \eps \right\}, x_0 \right), \frac{(1-\eps)\tilde{\eta}(\eps, x_0^A)}{2} \right\},
\]
where $\eta(\cdot, x_0) >0$ is chosen from Proposition \ref{prop:ABC}. Then we have that 
\[
\|P-P\circ P_A\| < \min \left\{ \frac{(1-\eps)\tilde{\eta} (\eps,x_0^A)}{2}, \eps \right\}.
\]
Thus, 
\begin{align*}
\|P\circ P_A (x_0)\| &= \|P(x_0) + (P_A (x_0) - P(x_0))\| \\
&> \left( 1 - \frac{(1-\eps)\tilde{\eta}(\eps, x_0^A)}{2} \right) - \| P - P_A\| \\
&> \left( 1 - \frac{(1-\eps)\tilde{\eta}(\eps, x_0^A)}{2} \right) - \frac{(1-\eps)\tilde{\eta}(\eps, x_0^A)}{2} = 1 - (1-\eps) \tilde{\eta}(\eps, x_0^A).
\end{align*}
This implies that there exists $\overline{Q} \in \mathcal{P}(^2 \ell_\infty^A, H)$ with $\| \overline{Q}\| =1$ such that 
\[
\|\overline{Q}(x_0^A)\|=1 \quad \text{and}\quad \| \overline{Q}-P\circ P_A \| < \eps,
\]
where $P\circ P_A$ is viewed as an element of $\mathcal{P}(^2 \ell_\infty^A, H)$. Define $Q \in \mathcal{P}(^2 c_0, H)$ as the natural extension of $\overline{Q}$ to $c_0$, that is, $Q(x)= \overline{Q}(x^A)$. Note that $\|Q\| = \|Q(x_0)\| = 1$ and 
\[
\|Q-P\| \leq \| \overline{Q}-P\circ P_A \| + \| P\circ P_A - P \| < 2\eps,
\]
which completes the proof.
\end{proof}

Notice that the hypothesis of $H$ being finite dimensional was only used to affirm that  the pair $(\ell_\infty^A, H)$ has the $2$-homogeneous polynomial $\Lpp$. If this held for infinite dimensional Hilbert spaces, so it would do Theorem \ref{c0-result}.

\subsection{Vector-valued polynomial $\Lpp$} In the previous subsection, we deduce some differentiability properties of projective (symmetric) tensor products from $\Lpp$ properties for scalar-valued polynomials and multilinear operators. In this subsection, we focus on the $N$-ho\-mo\-ge\-neous polynomial $\Lpp$ in the vector-valued case, although we cannot always get differentiability properties of tensor products from the vector-valued $\Lpp$ (see the comment below Corollary~\ref{property-beta}).
Recall that a Banach space $Y$ has the {\it property $\beta$} with constant $0 \leq \rho < 1$ if there exist $\{y_i: i \in I\} \subset S_Y$ and $\{ y_i^*: i \in I \} \subset S_{Y^*}$ such that
\begin{itemize}
\itemsep0.3em 
\item[(i)] $y_i^*(y_i) = 1$ for all $i \in I$,
\item[(ii)] $|y_i^*(y_j)| \leq \rho < 1$ for all $i, j \in I$ with $i \not= j$,
\item[(iii)] $\|y\| = \sup_{i \in I} |y_i^*(y)|$ for all $y \in Y$.
\end{itemize} 
Classic examples of Banach spaces satisfying the property $\beta$ are $c_0$ and $\ell_{\infty}$.
This property was introduced by Lindenstrauss in \cite{Lin}, in order to obtain examples of spaces $Y$ such that the set of norm attaining operators on $\mathcal{L}(X,Y)$ is dense in the whole space, for every Banach space $X$. In \cite[Proposition~2.8]{DKLM} it is proved that if $X$ is SSD and $Y$ has property $\beta$ then the pair $(X,Y)$ has the $\Lpp$. Next, we prove the polynomial version of this result.

\begin{proposition} \label{propertybeta} Let $X, Y$ be Banach spaces. If $Y$ has property $\beta$ and $(X, \K)$ has the $N$-homogeneous polynomial $\Lpp$, then $(X, Y)$ has the $N$-homogeneous polynomial $\Lpp$.
\end{proposition}

\begin{proof} Let $\e \in (0,1) $ and $x_0 \in S_X$ be fixed. Suppose that $P \in \mathcal{P}(^N X, Y)$ with $\|P\| = 1$ satisfies 
	\begin{equation*}
		\|P(x_0)\| > 1 - \min \left\{ \eta(\tilde{\e}, x_0), \tilde{\e}) \right\},
	\end{equation*}
	where $\eta$ is the one in the definition of the $N$-homogeneous polynomial $\Lpp$ for the pair $(X, \K)$ and
	\begin{equation*}
		\tilde{\e} := \left( \frac{1-\rho}{4} \right) \e > 0, 
	\end{equation*}
	where $0 \leq \rho < 1$ is the constant in the definition of property $\beta$. Let us take $\alpha_0 \in \Lambda$ such that 
	\begin{equation*}
		|(P^t y_{\alpha_0}^*)(x_0)| = |y_{\alpha_0}^* (P(x_0))| > 1 - \eta(\tilde{\e}, x_0).
	\end{equation*}
	Then, there exists $Q \in \mathcal{P}(^N X, \K)$ with $\|Q\| = 1$ such that 
	\begin{equation*}
		|Q(x_0)| = 1 \ \ \ \mbox{and} \ \ \ \| Q - P^t y_{\alpha_0}^* \| < \tilde{\e}.
	\end{equation*}	
	Let us define $\tilde{P}\colon X\to Y$ by
	\begin{equation*}
		\tilde{P}(x) := P(x) + \left( (1+ \e) Q - P^t y_{\alpha_0}^* \right)(x) y_{\alpha_0}
	\end{equation*}
	and note that $\| \tilde{P} - P \| < \e +  \tilde{\e}$. We will prove now that $\tilde{P}$ attains its norm at $x_0$. Notice first that for every $x \in X$, we have that 
	\begin{equation*}
		[\tilde{P}^t y_{\alpha_0}^*](x) = y_{\alpha_0}^* (\tilde{P}(x)) = (1+\e)Q(x), 
	\end{equation*}
	which shows that $\tilde{P}^t y_{\alpha_0}^* = (1 + \e)Q$. On the other hand, if $\alpha \not= \alpha_0$ and $x \in B_X$, then 
	\begin{eqnarray*}
		\| [ \tilde{P}^t y_{\alpha}^*](x)\| &\leq& \|P\| + |y_{\alpha}^*(y_{\alpha_0}) | \left( \e \|Q\| + \|Q - P^t y_{\alpha_0}^*\| \right) \\
		&<& 1 + \rho(\e + 2 \tilde{\e}) \\
		&<& 1 + \e. 
	\end{eqnarray*}
	This shows that $\|\tilde{P}\| = \| \tilde{P}^t y_{\alpha_0}^*\| = |y_{\alpha_0}^* (\widetilde{P}(x_0))|$, which implies that $\|\tilde{P}\| = \|\tilde{P}(x_0)\|$ as desired. Therefore, the pair $(X, Y)$ has the $N$-homogeneous polynomial $\Lpp$.	
\end{proof} 
As an immediate consequence of the previous proposition we have the following.
\begin{corollary} \label{property-beta} The following results hold true.
\begin{itemize}
  \item[(i)] If $X$ is a finite dimensional Banach space, then the pairs $(X, c_0)$ and $(X, \ell_{\infty})$ have the $N$-homogeneous polynomial $\Lpp$.
  \item[(ii)] The pairs $(\ell_2, c_0)$ and $( \ell_2, \ell_{\infty})$ have the $N$-homogeneous polynomial $\Lpp$.
    \item[(iii)] The pairs $(c_0, c_0)$ and $(c_0, \ell_{\infty})$ have the $2$-homogeneous polynomial $\Lpp$ in the complex case.
\end{itemize}
\end{corollary}

In view of the isometry
$((\sten \ell_2) \pten \ell_1)^* = \mathcal{P}(^N \ell_2, \ell_\infty)$, the results in Propositions~\ref{thm BPBpp for bilinear and polynomial} and \ref{local}, and the fact that $( \ell_2, \ell_{\infty})$ has the $N$-homogeneous polynomial $\Lpp$, it is natural to ask if $(\sten \ell_2) \pten \ell_1$ is SSD on the set of elementary tensors. However, it is easy to see that this is not possible, since the norm of $\ell_1$ is not SSD (see \cite[Theorem 7]{F} or \cite[Example 1.1]{GGS}). Although in general these notions can not be related in the vector valued case, next we show that when the codomain is a Banach space with micro-transitive norm, they do have a relation.

\begin{proposition}
Let $X$ be a Banach space and $Y$ a Banach space with micro-transitive norm. The pair $(X, Y^*)$ has the $N$-homogeneous polynomial $\Lpp$ if and only if $\sten X \pten Y$ is SSD on the set $V=\{\otimes^N x \otimes y:\,\, \|x\|=\|y\|=1\}$.
\end{proposition}
\begin{proof}
First we are going to show that the $N$-homogeneous polynomial $\Lpp$ property implies that the space $\sten X \pten Y$ is SSD on $V$. Analogously as we did in Proposition~\ref{thm BPBpp for a set is unif SSD} with the uniform strong subdifferentiability, it can be proved that $\sten X \pten Y$ is SSD on $V$ if and only if $(\sten X \pten Y, \mathbb{K})$ has the $\Lpp$ for the set $V$. 

Let $\eta$ be the one in the definition of the $N$-homogeneous polynomial $\Lpp$ for the pair $(X, Y^*)$, $\tilde{\eta}$ the one in the definition of the BPBpp for the pair $(Y, \mathbb{K})$ and $\beta(\e)$ the one in the definition of micro-transitivity property of $Y^*$ (see, for instance, \cite[Proposition 3.4]{CDKKLM}). Given $\e$, let $\varphi \in (\sten X \pten Y)^*$, $\|\varphi\|=1$, and $\otimes^N x_0\otimes y_0 \in V$ be such that
\begin{equation}\label{eq vector valued Lpp}
\varphi(\otimes^N x_0\otimes y_0)>1-\min\left\{\eta\left(\frac{\beta(\e)}{2}, x_0\right), \tilde{\eta}\left(\frac{\beta(\e)}{2}\right)\right\}.
\end{equation}
We want to show that there is $\psi \in (\sten X \pten Y)^*$, $\|\psi\|=1$, with $$\psi(\otimes^N x_0 \otimes y_0)=1\quad \text{and}\quad \|\psi-\varphi\|<2\e.
$$
On the one hand, in virtue of the duality $((\sten X) \pten Y)^* = \mathcal{P}(^N X, Y^*)$, there is a norm-one polynomial $P\in \mathcal{P}(^N X, Y^*)$ such that
$
P(x)(y)=\varphi(\otimes^N x\otimes y).
$
By \eqref{eq vector valued Lpp} we have that \linebreak $P(x_0)(y_0)>1- \tilde{\eta}\left(\frac{\beta(\e)}{2}\right)$ and, since the pair $(Y, \mathbb{K})$ has the BPBpp, there exist $y_0^*\in S_{Y^*}$ such that
$$
y_0^*(y_0)=1\quad \text{and}\quad \|y_0^*-P(x_0)\|<\frac{\beta(\e)}{2}.
$$
On the other hand, given that $\|P(x_0)\|>1-\eta\left(\frac{\beta(\e)}{2}, x_0\right)$, there exist a norm-one polynomial $Q\in \mathcal{P}(^N X, Y^*)$ such that
$$
\|Q(x_0)\|=1\quad \text{and}\quad \|Q-P\|<\frac{\beta(\e)}{2}.
$$
In sum, we have $\|y_0^*-Q(x_0)\|<\beta(\e)$ and this implies that there exist a surjective isometry $T\colon Y^*\to Y^*$ such that $T(Q(x_0))=y_0^*$ and $\|T-\id_{Y^*}\|<\e$. Finally, define $\tilde{Q}\colon X\to Y^*$ by \linebreak $\tilde{Q}(x)=T(Q(x))$ and note that
$$
\tilde{Q}(x_0)(y_0)=1\quad \text{and}\quad \|\tilde{Q}-P\|\leq \|\tilde{Q}-Q\|+\|Q-P\|<2\e.
$$
Thus, if $\psi \in (\sten X \pten Y)^*$, $\|\psi\|=1$, is the linear functional associated to the polynomial $\tilde{Q}$, we have
$$
\psi(\otimes^N x_0 \otimes y_0)=1\quad \text{and}\quad \|\psi-\varphi\|<2\e,
$$
which is the desired statement.

Now let us suppose that the pair $(X, Y^*)$ does not have the $N$-homogeneous polynomial $\Lpp$. We want to see that $(\sten X \pten Y, \mathbb{K})$ does not have the $\Lpp$ for the set $V$. By hypothesis, there is $x_0\in S_{X}$, $\e>0$ and $(P_j)_j\subseteq \mathcal{P}(^N X, Y^*)$ norm one polynomials such that
\begin{equation}\label{eq aux}
\|P_j(x_0)\|\rightarrow 1 \,\,\,\text{ and } \,\,\, \dist\left(P_j,\{ P\in S_{\mathcal{P}(^N X, Y^*)}: \|P(x_0)\|=1\}\right) > \e.
\end{equation}
Composing each $P_n$ with an adequate isometry $T_n:Y^* \rightarrow Y^*$, we may assume that $P_n(x_0)$ is a multiple of a fixed $y_0^* \in S_{Y^*}$. For each $j$ define $\varphi_j \in (\sten X \pten Y)^*$ as
$$\varphi_j(z)=\langle y_0, P_j (z)\rangle.$$
Then $\Vert \varphi_j \Vert=\Vert P_j\Vert=1$, $|\varphi_j(\otimes^N x_0 \otimes y_0)|=\|P_j (x_0)\|\rightarrow 1$, and equation \eqref{eq aux} implies that 
$$\dist\left(\varphi_j,D(\otimes^N x_0 \otimes y_0)\right) > \e.$$
Therefore, $(\sten X \pten Y, \mathbb{K})$ does not have the $\Lpp$ for the set $V$, as we wanted to see.
\end{proof}

As a consequence we obtain the following corollary.

\begin{corollary} If $H$ is a Hilbert space, then the following results hold true.
\begin{itemize}
  \item[(i)] The space $(\sten H )\pten H$ is SSD on the set $V=\{\otimes^N x \otimes y:\,\, \|x\|=\|y\|=1\}$.
  \item[(ii)] If, in addition, $H$ is complex and finite dimensional, the space $({\ensuremath{c_0 \widehat{\otimes}_{\pi_s}}} c_0)\pten H$ is SSD on the set $V=\{\otimes^2 x \otimes y:\,\, \|x\|=\|y\|=1\}$.
\end{itemize}
\end{corollary}

\subsection{A negative result on bilinear symmetric forms}

It is a well known fact that the polarization formula gives an isomorphism between the space of $N$-homogeneous polynomials $\mathcal{P} (^NX, Z)$ and the space of $N$-linear symmetric mappings $\mathcal{L}_s (^NX, Z)$. Moreover, in some spaces this isomorphism is in fact an isometry. This is the case when $X$ is a Hilbert space. Then, it is natural to ask if it is possible to obtain similar results to the ones obtained before when we deal with symmetric forms instead of polynomials. In this short subsection we will show, with a simple counterexample, that in Proposition \ref{thm BPBpp for bilinear and polynomial} (ii) we cannot replace the $N$-homogeneous polynomial BPBpp by a similar property using $N$-linear symmetric BPBpp.

Let us begin with some proper definitions and remarks. We say that the pair $(X,Z)$ has the \emph{$N$-linear symmetric Bishop-Phelps-Bollobás point property} ($N$-linear symmetric BPBpp, for short) if given $\e > 0$, there exists $\eta(\e) > 0$ such that whenever $A \in \mathcal{L}_s (^NX, Z)$, $\|A\| = 1$, and $(x_1, \ldots, x_N)\in S_X\times \cdots\times S_X$ satisfy
\begin{equation*}
\left\|A(x_1, \ldots, x_N) \right\| > 1 - \eta(\e),
\end{equation*}
there exists $B \in \mathcal{L}_s (^NX, Z)$ with $\|B\| = 1$ such that 
\begin{equation*} 	
\left\|B(x_1, \ldots, x_N)\right\| = 1 \ \ \ \ \mbox{and} \ \ \ \ \|B - A\| < \e.
\end{equation*} 

When dealing with symmetric multilinear forms, we have the linear isometry 
$$
\mathcal{L}_s(^NX)=\left( \psten X\right)^*,
$$
where we endow the $N$-fold symmetric tensor product with the (full, not symmetric) projective norm $\pi$. Thus, it is reasonable to wonder if an analogous to Proposition \ref{thm BPBpp for bilinear and polynomial} holds, and we can relate the $N$-linear symmetric BPBpp with USSD on the set $U_s=\{\otimes ^N x:\,\, \|x\|=1\}$ (considering projective norm $\pi$). By Theorem \ref{theoremC}, $\sten \ell_2$ is USSD on the set $U_s$ (recall that for Hilbert spaces the projective norm and the symmetric projective norm coincide). But, as we will show bellow, $\ell_2$ does not enjoy the BPBpp for $N$-linear symmetric mappings. Therefore, a proposition similar to \ref{thm BPBpp for bilinear and polynomial} replacing polynomials by symmetric multilinear mapping can not be obtained.

\begin{example}\label{counterexample}
The pair $(\ell_2, \K)$ fails the bilinear symmetric BPBpp. Moreover, the pair $(\ell_2^2, \K)$ fails this property.
\end{example}
\begin{proof}
Suppose, on the contrary, that $(\ell_2^2, \K)$ has the bilinear symmetric BPBpp and consider $A:\ell_2^2 \times \ell_2^2\rightarrow \K$ the symmetric bilinear form given by the matrix
$$
\left(\begin{array}{lr}
1 	& 	 0\\
0	&	1		\\
\end{array}\right).$$
Given $0<\varepsilon<1$, let $0<\eta(\varepsilon)<1$ be the one in the definition of the bilinear symmetric BPBpp. Let $a,b>0$ be such that $a^2+b^2=1$ and
$$A((a,b),(a,-b))=a^2-b^2>1-\eta(\varepsilon).$$
Then, there is a symmetric norm-one bilinear form $B$ such that $$|B((a,b),(a,-b))|=1\quad \text{and}\quad \Vert A- B\Vert <\varepsilon.$$
Now, let
$$
\left(\begin{array}{lr}
d_1 	& 	 d_3	 \\
d_3	&	d_2		\\
\end{array}\right)$$
be the matrix associated to $B$. 
Since $\Vert B\Vert=1$, we have that $|d_1|$ and $|d_2|$ cannot exceed 1. Therefore,
$$ |d_1\, a^2-d_2\,b^2|=|B((a,b),(a,-b))|=1= a^2+b^2$$
imply that $d_1=-d_2$ and $|d_1|=|d_2|=1$. Then, $\Vert A- B\Vert \geq 1$, which is the desired contradiction.
\end{proof}

In contrast with this negative result, it is worth mentioning that complex Hilbert spaces have the BPBpp for several classes of operators: self-adjoints, anti-symmetric, unitary, normal, compact normal, compact and Schatten-von Neumann operators (see \cite[Theorem 3.1]{CDJ}). It also has the Bishop-Phelps-Bollob\'as property for symmetric bilinear mappings and hermitian bilinear mappings and, in the real case, it has the Bishop-Phelps-Bollob\'as property for symmetric bilinear mapping (see \cite{GLM}). Also, although Hilbert spaces fail to have  the bilinear symmetric BPBpp, the pair $(\ell_2^d, \K)$ enjoys a local Bishop-Phelps-Bollob\'as type property for symmetric bilinear forms by compactness.

 

\noindent 
\textbf{Acknowledgements:} The authors are thankful to Gilles Godefroy for suggesting the topic of the paper. Also they would like to thank Richard Aron, Petr Hájek, Sun Kwang Kim and Abraham Rueda Zoca for fruitful conversations about some specific parts of the paper during the writing procedure. 

\noindent 
\textbf{Funding information:} The first author was supported by the Spanish AEI Project PID2019 - 106529GB - I00 / AEI / 10.13039/501100011033 and also by Spanish AEI Project PID2021-122126NB-C33 / MCIN / AEI / 10.13039 / 501100011033 (FEDER). The second was supported by National Research Foundation of Korea (NRF-2019R1A2C1003857), by POSTECH Basic Science Research Institute Grant (NRF-2021R1A6A1A10042944) and by a KIAS Individual Grant (MG086601) at Korea Institute for Advanced Study. The third author was supported by  CONICET PIP 11220130100329CO. Finally, the forth was supported by CONICET PIP 11220200101609CO and  ANPCyT PICT 2018-04250.





\begin{thebibliography}{FaHaMo}


	
\bibitem{A} \textsc{M.D.~Acosta}, 
{\it The Bishop-Phelps-Bollob\'as property for operators on $\mathcal{C}(K)$}. 
Banach J. Math. Anal. \textbf{10} (2016), pp. 307--319. 

\bibitem{AAGM} \textsc{M.D.~Acosta, R.M.~Aron, D.~Garc\'ia and M.~Maestre}, 
{\it The Bishop-Phelps-Bollob\'as theorem for operators}.
J. Funct. Anal. {\bf 294} (2008), pp. 2780--2899.


\bibitem{Acoet6} \textsc{M.D.~Acosta, J.~Becerra-Guerrero, Y. S.~Choi,  D.~Garc\'ia, S.K.~Kim, H.J.~Lee and M.~Maestre}, 
{\it The Bishop-Phelps-Bollob\'as property for bilinear forms and polynomials}.
J. Math. Soc. Japan {\bf 66} (2014), pp. 957--979.

\bibitem{ABGM} \textsc{M.D.~Acosta, J.~Becerra-Guerrero, D.~Garc\'ia and M.~Maestre}, 
{\it The Bishop-Phelps-Bollob\'as Theorem for bilinear forms}. 
Trans. Amer. Math. Soc. {\bf 365} (2013), pp. 5911--5932.


\bibitem{Ale} \textsc{R.~Alencar}, 
{\it An application of Singer's theorem to homogeneous polynomials}. 
Contemp. Math. {\bf 144} (1993), pp. 1--8.

\bibitem{Alt} \textsc{Z.~Altshuler}, 
{\it Uniform convexity in Lorentz sequence spaces}. 
Israel J. Math. {\bf 20} (1975), pp. 260--274.

	
\bibitem{AOPR} \textsc{A.~Aparicio, F.~Oca\~na, R.~Pay\'a, and A.~Rodr\'iguez}, 
{\it A non-smooth extension of Fr\'echet-differentiability of the norm with applications to numerical ranges}. 
Glasgow Math. J. {\bf 28} (1986), pp. 121--137.


\bibitem{ABC} \textsc{R. Aron, C. Boyd and Y.S. Choi}, \emph{Unique Hahn-Banach theorems for spaces of homogeneous polynomials}, J. Austr. Math. Soc. \textbf{70} (2001), 387--400.
 
\bibitem{AroHerVal} \textsc{R.M. Aron, C. Herv\'es and M. Valdivia}, 
{\it Weakly continuous mappings on Banach spaces}.
J. Funct. Anal. {\bf 52} (1983), pp. 189--204.
p. pp. 124046

\bibitem{BP} \textsc{E.~Bishop and R.R.~Phelps}, 
{\it A proof that every Banach space is subreflexive}.
Bull. Amer. Math. Soc. {\bf 67} (1961), pp. 97--98

\bibitem{Bol} \textsc{B.~Bollob\'as}, 
{\it An extension to the theorem of Bishop and Phelps}.
Bull. London. Math. Soc. {\bf 2} (1970), pp. 181--182.


\bibitem{BecRod} \textsc{J.~Becerra-Guerrero and A.~Rodr\'iguez-Palacios}, 
{\it Subdifferentiability of the norm on JB-Triples}.
Quart. J. Math. {\bf 54} (2003), pp. 381--390.


	
\bibitem{CDKKLM} \textsc{F. Cabello S\'anchez, S.~Dantas, V.~Kadets, S.~K.~Kim, H.~J.~Lee and M.~Mart\'in}, 
{\it On Banach spaces whose group of isometries acts micro-transitively on the unit sphere}.
J. Math. Anal. Appl. {\bf 488} (2020), pp. 1--14.


\bibitem{CasLin} \textsc{P.G.~Casazza and B.L.~Lin}, 
{\it Some geometric properties of Lorentz sequence spaces}.
The Rocky Mountain Journal of Mathematics {\bf 7} (1977), pp. 683--698.




	
{\it Some Bishop-Phelps-Bollob\'as type properties in Banach spaces with respect to minimum norm of bounded linear operators}. 
Ann. Funct. Anal. {\bf 12} (2021), pp. 1--15.

\bibitem{Che}\textsc{S.T. Chen},
{\it Geometry of Orlicz Spaces}.
Dissert. Math. {\bf 356} (1996).

\bibitem{ChoKim} \textsc{G.~Choi and S.K.~Kim}, 
{\it The Bishop–Phelps–Bollob\'as Property on the Space of $c_0 $-Sum}.
Mediterranean J. Math., {\bf 19} (2022), pp. 1--16.

	
\bibitem{CDJ} \textsc{Y.S. Choi, S. Dantas, and M. Jung},
{\it The Bishop-Phelps-Bollobás properties in complex Hilbert spaces}. 
Math. Nachr. {\bf 294} (2021), pp. 2105--2120.
	
\bibitem{Contreras} \textsc{M.D.~Contreras}, 
{\it Strong subdifferentiability in spaces of vector-valued continuous functions}.
Quart. J. Math. Oxford Ser. {\bf 47} (1996), pp. 147--155.
	
	
\bibitem{CP} \textsc{M.D.~Contreras and R.~Pay\'a}, 
{\it On upper semicontinuity of duality mappings}.
Proc. Amer. Math. Soc. {\bf 121} (1994), pp. 451--459.
	
	
	
\bibitem{D} \textsc{S.~Dantas}, 
{\it Some kind of Bishop-Phelps-Bollob\'as property}.
Math. Nachr. {\bf 290} (2017), pp. 774--784.
		
	
\bibitem{DGJR} \textsc{S. Dantas, L.C. Garc\'ia-Lirola, M. Jung, and A. Rueda Zoca}, 
{\it On norm-attainment in (symmetric) tensor products}. 
Quaestiones Mathematicae. (2022) DOI: 10.2989/16073606.2022.2032862.

\bibitem{DJRR} \textsc{S.~Dantas, M.~Jung, \'O.~Rold\'an, and A.~Rueda-Zoca}, 
\textit{Norm-attaining tensors and nuclear operators}. 
Mediterr. J. Math. \textbf{19}, 38 (2022).



\bibitem{DKKLM1} \textsc{S.~Dantas, V.~Kadets, S.K.~Kim, H.J.~Lee, M.~Mart\'{\i}n}, \textit{On the pointwise Bishop-Phelps-Bollob\'as property for operators}, {Canad. J. Math.} \textbf{71} (2019), no. 6, 1421--1443.

\bibitem{DKL} \textsc{S.~Dantas, S.K.~Kim, and H.J.~Lee}, 
{\it The Bishop-Phelps-Bollobás point property}.
J. Math. Anal. Appl. {\bf 444} (2016), pp. 1739--1751.


	

\bibitem{DKLM} \textsc{S.~Dantas, S.K.~Kim, H.J.~Lee and M.~Mazzitelli}, 
{\it Local Bishop-Phelps-Bollob\'as properties}. 
J. Math. Anal. Appl. \textbf{468} (2018), pp. 304--323.

\bibitem{DKLM2} \textsc{S.~Dantas, S.K.~Kim, H.J.~Lee and M. ~Mazzitelli}, 
{\it Strong subdifferentiability and local Bishop-Phelps-Bollob\'as properties}.
Rev. R. Acad. Cienc. Exactas F\'is. Nat. Ser. A Mat. RACSAM, {\bf 114} (2020),  pp. 1--16.


\bibitem{DR}\textsc{S.~Dantas and A.~Rueda Zoca}, 
{\it A characterization of a local vector valued Bollobás theorem}. 
Results Math. {\bf 76} (2021), pp. --14.

\bibitem{defant1992tensor}
\textsc{A. Defant and K. Floret},
{\it Tensor norms and operator ideals}.
Amsterdam: North-Holland (1993).

\bibitem{DGZ}
\textsc{R. Deville, G. Godefroy, and V. Zizler},
{\it Smoothness and renormings in Banach spaces}.
Pitman Monographs and Surveys in Pure and Applied Mathematics Longman Scientific Technical, Harlow; copublished in the United States with John Wiley Sons, Inc., New York (1993).

\bibitem{diestel2008metric}
\textsc{J. Diestel, J.H. Fourier and J. Swart},
{\it The Metric Theory of Tensor Products. Grothendieck's R\'esum\'e Revisited}.
American Mathematical Society (2008).


\bibitem{DimZal}
\textsc{V. Dimant and I. Zalduendo},
{\it Bases in spaces of multilinear forms over Banach spaces}.
J. Math. Anal. Appl. {\bf 200} (1996), pp. 548--566.

\bibitem{Din} \textsc{S.~Dineen}, 
{\it Complex Analysis on Infinite Dimensional Spaces}.
Springer Verlag, 1999.


\bibitem{floret1997natural}
\textsc{K.~Floret},
{\it Natural norms on symmetric tensor products of normed spaces}.
Note di Matematica {\bf 17} (1997), pp. 153--188.

\bibitem{F}
\textsc{C.~Franchetti},
{\it Lipschitz maps and the geometry of the unit ball in normed spaces}.
Arch. Math. {\bf 46} (1986), pp. 76--84.


\bibitem{FP} \textsc{C.~Franchetti and R.~Pay\'a}, 
{\it Banach spaces with strongly subdifferentiable norm}.
Boll. Uni. Mat. Ital. {\bf VII-B} (1993), pp. 45-70.


\bibitem{GLM} \textsc{D. Garc\'ia, H.J. Lee and M. Maestre}, 
{\it The Bishop-Phelps-Bollob\'as property for hermitian forms on Hilbert spaces}.
Quart. J. Math. {\bf 65} (2014), pp. 201--209.

\bibitem{GGS} \textsc{J.R.~Giles, D.A.~Gregory, and B.~Sims}, 
{\it Geometrical implications of upper semicontinuity of the duality mapping of a Banach space}.
Pacific J. Math. {\bf 79} (1978), pp. 9--109.




\bibitem{Godefroy} \textsc{G.~Godefroy}, 
{\it Some applications of Simons' inequality}.
Serdica Math. J. {\bf 26} (2000), pp. 59--78.



\bibitem{GMZ} \textsc{G.~Godefroy, V.~Montesinos, and V.~Zizler}, 
{\it Strong subdifferentiability of norms and geometry of Banach spaces}.
Comment. Math. Univ. Carolin. {\bf 36} (1995), pp. 493--502.

\bibitem{Gon} \textsc{R.~Gonzalo}, 
{\it Upper and lower estimates in Banach sequence spaces}.
Commentationes Mathematicae Universitatis Carolinae {\bf 36} (1995), pp. 641--653.

\bibitem{GJ} \textsc{R.~Gonzalo and J.A.~Jaramillo}, 
{\it Compact polynomials between Banach spaces}.
Proceedings of the Royal Irish Academy. Section A: Mathematical and Physical Sciences {\bf 95A} (1995), pp. 213--226.

\bibitem{Gregory} \textsc{D.A.~Gregory}, 
{\it Upper semi-continuity of subdifferential mappings}.
Can. Math. Bull. {\bf 23} (1980), pp. 11--19.



\bibitem{JarMor} \textsc{J.A.~Jaramillo and L.A. Moraes}, 
{\it Duality and reflexivity in spaces of polynomials}.
Arch. Math. {\bf 74} (2000), pp. 82--293.

\bibitem{JarPriZal} \textsc{J.A.~Jaramillo, A.~Prieto and I.~Zalduendo}, 
{\it The bidual of the space of polynomials on aBanach space}.
Math. Proc. Camb. Phil. Soc. {\bf 122} (1997), pp. 457--471.


\bibitem{J} \textsc{M.~Jung}, 
{\it Daugavet property of Banach algebras of holomorphic functions and norm-attaining holomorphic functions}.
preprint available on https://arxiv.org/abs/2105.03967.

\bibitem{KL} \textsc{S.K.~Kim and H.J.~Lee}, 
{\it Uniform convexity and the Bishop-Phelps-Bollob\'as property}.
Canad. J. Math. {\bf 66}, (2014), pp. 373--386.

\bibitem{Lin}
J. Lindenstrauss,
{\it On operators which attain their norm}.
Israel J. Math. {\bf 1}, (1963), pp. 139--148.


\bibitem{LTII}
J. Lindenstrauss and L. Tzafriri,
{\it On Orlicz sequence spaces II}.
Israel J. Math. {\bf 11} (1972), pp. 355--379.


\bibitem{LTIII}
J. Lindenstrauss and L. Tzafriri,
{\it On Orlicz sequence spaces III}.
Israel J. Math. {\bf 14} (1973), pp. 368--389.
 
\bibitem{LT1}
J. Lindenstrauss and L. Tzafriri,
{\it Classical Banach spaces I. Sequence spaces.}
Ergebnisse der Mathematik und ihrer Grenzgebiete. 92.
 Berlin-Heidelberg-New York: Springer-Verlag. XIII, 190 p. , 1977. 


\bibitem{Muj} \textsc{J. Mujica}, 
{\it Reflexive spaces of homogeneous polynomials}.
Bulletin Polish Acad. Sci. Math. {\bf 49} (2001), pp. 211--222.

\bibitem{Rod} \textsc{A.~Rodr\'iguez-Palacios}, 
{\it A numerical range characterization of uniformly smooth Banach spaces}.
Proc. Amer. Math. Soc. {\bf 129} (2001), pp. 815--821.

\bibitem{ryan2002introduction} \textsc{R. Ryan}, 
{\it Introduction to Tensor Products of Banach Spaces.}
Springer Monographs in Mathematics (2012).


\bibitem{Sain} \textsc{D.~Sain}, 
{\it Smooth points in operator spaces and some Bishop-Phelps-Bollob\'as type theorems in Banach spaces}.
Oper. Matrices {\bf 13} (2019), pp. 433--445.


\bibitem{Smith} \textsc{M.~A.~Smith}, 
{\it Some examples concerning rotundity in Banach spaces}.
Math. Ann. {\bf 233} (1978), pp. 155--161.

\bibitem{Smu} \textsc{V.L.~\v{S}mulyan}, 
{\it Sur le d\'erivabilit\'e de la norme dans l'espace de Banach}.
C.R. Acad. Sci. URSS (Doklady) {\bf 13} (1940), pp. 643--648.

\bibitem{Sza} \textsc{A.~Szankowski}, 
{\it Subspaces without the approximation property}.
Israel J. Math. {\bf 30} (1978), pp. 123-129.


\bibitem{W} \textsc{G.~Willis}, 
{\it The compact approximation property does not implies the approximation property}.
Studia Math. {\bf 103} (1992), pp. 99--108	
	

 
\end{thebibliography}
\end{document}